\documentclass{article}
\usepackage{graphicx} 
\usepackage{amsthm} 
\usepackage{amssymb}
\newtheorem{theorem}{Theorem}
\newtheorem{corollary}{Corollary}
\newtheorem{lemma}{Lemma}
\newtheorem{proposition}{Proposition}

\newtheorem{definition}{Definition}

\usepackage{amsmath}
\usepackage{float}
\usepackage{subfigure}
\usepackage{tikz}
\usepackage{titlesec}
\titlelabel{\thetitle.\quad}
\usepackage{comment}

\title{Clean Graphs and Idempotent Graphs over Finite Rings: An Approach Based on $\mathbb{Z}_n$}
\author{Felicia Servina Djuang$^{1}$, Indah Emilia Wijayanti$^{2}$, and Yeni Susanti$^{3}$\\
{\small $^{1,2,3}$Department of Mathematics, Universitas Gadjah Mada, Yogyakarta, Indonesia}\\
\small{$^{1}$feliciadjuang25@mail.ugm.ac.id, $^{2}$ind\_wijayanti@ugm.ac.id, $^{3}$yeni\_math@ugm.ac.id}}
\date{}

\begin{document}
\maketitle
\begin{abstract}
Let $R$ be a finite ring with identity. The idempotent graph $I(R)$ is the graph whose vertex set consists of the non-trivial idempotent elements of $R$, where two distinct vertices $x$ and $y$ are adjacent if and only if $xy = yx = 0$. The clean graph $Cl(R)$ is a graph whose vertices are of the form $(e, u)$, where $e$ is an idempotent element and $u$ is a unit of $R$. Two distinct vertices $(e,u)$ and $(f, v)$ are adjacent if and only if $ef = fe = 0$ or $uv = vu = 1$. The graph $Cl_2(R)$ is the subgraph of $Cl(R)$ induced by the set $\{(e, u) : e \text{ is a nonzero idempotent element of } R\}$. In this study, we examine the structure of clean graphs over $\mathbb{Z}_{n}$ derived from their $Cl_2$ graphs and investigate their relationship with the structure of their idempotent graphs. 
\end{abstract}

\section{Introduction}

    The study of zero-divisor graph and algebraic graphs over ring was first introduced by Beck in \cite{beck}, who considered all elements of a commutative ring $R$ as vertices and explored the structure of its zero-divisor graph, primarily focusing on graph colorings. In \cite{anderson}, Anderson and Livingston formally defined the zero-divisor graph of commutative ring $R$, denoted by $\Gamma(R)$. In their definition, the vertices of $\Gamma(R)$ consist of all nonzero zero-divisors of $R$, with two distinct vertices $x$ and $y$ connected by an edge if and only if $xy=0$. 
    
    In addition to the zero-divisor graph, there also exists the unit graph over the ring $\mathbb{Z}_n$, introduced by Grimaldi in \cite{grimaldi}, as well as the idempotent graph introduced by Akbari et al. in \cite{akbari}. The vertex set of the unit graph over $\mathbb{Z}_n$ is $\mathbb{Z}_n$, and two vertices $x$ and $y$ are adjacent if $x + y$ is a unit in $\mathbb{Z}_n$. The vertex set of the idempotent graph over a ring $R$ consists of the nontrivial idempotent elements of $R$, and two vertices $x$ and $y$ are adjacent if $xy = yx = 0$. Further research on idempotent graphs over matrix rings was conducted by \cite{patil}, who determined the structure of the idempotent graph over the ring $M_{2}(\mathbb{F})$, where $\mathbb{F}$ is a field. Then, the study of graphs associated with rings has become an active area of research (see, for example, \cite{grimaldi},\cite{pongthana},\cite{singhpatekar}). The concepts of idempotent graphs and clean rings motivated Habibi et al. in \cite{habibiyet} to define a clean graph over a ring $R$, denoted by $Cl(R)$, where the vertices consist of all pairs of idempotent elements and unit elements of $R$. Two vertices $(e,u)$ and $(f,v)$ are adjacent if and only if $ef=fe=0$ or $uv=vu=1$. This naturally raises the question of how the clean graph of a ring is related to its idempotent graph, which served as the primary motivation for defining the clean graph.

    Research on algebraic graphs has potential applications in coding theory. \cite{fish} constructed linear codes derived from the incidence matrix of the line graph of the Hamming graph. This motivated \cite{jain} to construct linear codes based on the incidence matrix of the unit graph over $\mathbb{Z}_n$, by categorizing cases based on the number of prime factors of $n$.  

    In this context, the theory of clean graphs can be utilized for further studies in coding theory. However, additional research is needed, particularly focusing on the construction of linear codes. This requires a thorough investigation of the structure of clean graphs over finite rings, especially $\mathbb{Z}_n$. 
    
    Before proceeding, we recall some basic terminology that will be used in this paper. Let $R$ be a ring with identity. An element $e \in R$ is called \textit{idempotent} if $e^2=e$, and an element $u \in R$ is called a unit if there exists $v \in R$ such that $uv=vu=1$. The sets of all idempotent elements and all unit elements of $R$ are denoted by $Id(R)$ and $U(R)$, respectively. Additionally, the set $U(R)$ can be partitioned as follows:
    \begin{align*}
        U'(R)=\{u \in U(R): u^2=1\} \text{ and } U''(R) =U(R) \setminus U'(R).
    \end{align*}
    If an element $a \in R$ can be expressed as $a = e + u$, where $e$ is an idempotent and $u$ is a unit in $R$, then $a$ is called clean. A ring $R$ is said to be clean if every element of $R$ is clean. In \cite{immormino}, Immormino proved that every finite ring is clean. For any undefined notation or terminology in ring theory and further studies related to clean rings, we refer the reader to \cite{malikmor} and \cite{nicholzhou}.

    Let $G=(V(G),E(G))$ be a graph, where $V(G)$ and $E(G)$ represent the set of vertices and edges, respectively. A graph $G$ is said to be \textit{connected} if there exists a path between every pair of distinct vertices, and \textit{complete} if every pair of distinct vertices is adjacent. The complete graph with $n$ vertices is denoted by $K_n$. The degree of a vertex $v$ in graph $G$, denoted by $\deg_G(v)$, refers to the number of edges in graph $G$ that incident to $v$. Two graphs $G_1$ and $G_2$ are said to be isomorphic, denoted $G_1 \cong G_2$, if there exists a bijection $f: V(G_1) \to V(G_2)$ between their vertex sets such that for every pair of vertices $u,v \in V(G_1)$, $u$ and $v$ are adjacent in $G_1$ if and only if $f(u)$ and $f(v)$ are adjacent in $G_2$. For additional background on graph theory and relevant terminology, we refer the reader to \cite{wilson}.

    In \cite{habibiyet}, Habibi et al. defined the subgraphs $Cl_1(R)$ and $Cl_2(R)$ as the induced subgraphs of $Cl(R)$, where the vertex sets are given by
    \begin{align*}
        Cl_1(R)=\{(0,u): u \in U(R)\} \text{ and } Cl_2(R)=\{(e,u): e \in Id(R)\setminus\{0\}, u \in U(R)\}.
    \end{align*}
    Moreover, Remark 2.6 in \cite{habibiyet} provides a formula for the degree for any vertex $x=(e,u)$ in the graph $Cl_2(R)$ as follows: 
    \begin{align*}
        deg_{Cl_2(R)}(x)=\begin{cases}
            |Id(R)| + O_e|U(R)| -2, &\text{if } u \in U'(R),\\
            |Id(R)| + O_e|U(R)| -1, &\text{if } u \in U''(R),
        \end{cases}
    \end{align*}
    where $O_e:= |\{f \in Id(R)\setminus\{0\}: ef=fe=0\}|$. However, a counterexample can be found in the clean graph $Cl_2(\mathbb{Z}_{10})$. Specifically, the vertex $(6,3)$ has a degree of $6$, whereas applying the given formula in the remark results in a degree of $7$. To address this discrepancy, this paper provides a correction to the degree formula for vertices in $Cl_2(R)$. Additionally, we investigate the structure of the graphs $Cl_2(\mathbb{Z}_n)$, focusing on $Cl_2(\mathbb{Z}_{p^n})$ and $Cl_2(\mathbb{Z}_{p^nq^m})$. Furthermore, we explore the structures of $Cl_2(\mathbb{Z}_{p_1^{n_1}p_2^{n_2}p_3^{n_3}})$ and $Cl_2(\mathbb{Z}_{p_1^{n_1}p_2^{n_2}p_3^{n_3}p_4^{n_4}})$, from which we derive a generalization for the structure of $Cl_2(\mathbb{Z}_n)$ and find a relation about the structure of clean graph $Cl_2(R)$ with their idempotent graph $I(R)$, for arbitrary ring $R$.
    
\section{Results}
The following theorem provides a correction to the vertex degree formula in the graph $Cl_2(R)$.
\begin{theorem}\label{theorem_degree}
    Let $R$ be a ring with identity. For every $x=(e,u) \in V(Cl_2(R))$ we have
    $$deg_{Cl_2(R)}(x)=\begin{cases}
        \left|Id(R)\right| + O_e \left( \left| U(R)\right|-1 \right) -2, &\text{if } u \in U'(R),\\ 
        \left|Id(R)\right| + O_e \left( \left| U(R)\right|-1 \right) -1, &\text{if } u \in U''(R).\\ 
    \end{cases}$$
\end{theorem}
\begin{proof}
    For every $x=(e,u) \in V(Cl_2(R))$ we have two cases as follows.
    \begin{itemize}
        \item [(i)] If $u \in U'(R)$, then $u^2=1$. For every vertex $(f,v) \in V(Cl_2(R))$,  $$(f,v)(e,u) \in E(Cl_2(R)) \iff fe=ef=0 \text{ or } uv=1.$$ 
        \begin{itemize}
            \item[(i.a)] Suppose $uv=1$, it means $v=u$, so $(f,u)(e,u) \in E(Cl_2(R)$ for every $f \in Id(R)\setminus \{0,e\}$. There are $\left| Id(R)\right|-2$ possibilities. 
            \item[(i.b)] Suppose $fe=ef=0$, then $(f,v)(e,u) \in E(Cl_2(R))$ for every $v \in U(R)$. But, vertex $(f,u)$ has been included in the previous case, so there are $O_e\left(|U(R)|-1\right)$ remaining possibilities.
        \end{itemize}
       \item [(ii)] If $u \in U''(R)$, then $u^2 \neq 1$. For every vertex $(f,v) \in V(Cl_2(R))$,  $$(f,v)(e,u) \in E(Cl_2(R)) \iff fe=ef=0 \text{ or } uv=1.$$ 
        \begin{itemize}
            \item[(ii.a)] Suppose $uv=1$, it means $v \neq u$, so $(f,v)(e,u) \in E(Cl_2(R)$ for every $f \in Id(R) \setminus \{0\}$. There are $\left| Id(R)\right|-1$ possibilities. 
            \item[(ii.b)] Suppose $fe=ef=0$, then $(f,v)(e,u) \in E(Cl_2(R))$ for every $v \in U(R)$. But, vertex $(f,u^{-1})$ has been included in the previous case, so there are $O_e\left(|U(R)|-1\right)$ remaining possibilities.
        \end{itemize} \end{itemize}
\end{proof}

\subsection{Structure of Clean Graphs over $\mathbb{Z}_{p^n}$ and $\mathbb{Z}_{p^nq^m}$}

The following key lemma is instrumental in determining the structure and cha\-racteristics of the clean graph over $\mathbb{Z}_{p^n}$ and $\mathbb{Z}_{p^nq^m}$. 
\begin{lemma}\label{lemma_U'U''(R)}
    Given ring $\mathbb{Z}_{p^n}$ with a prime number $p$ and a natural number $n$. Let $a \in \mathbb{Z}_{p^n}$.
    \begin{itemize}
        \item[(i)] If $p \neq 2$, then
        $$a^2 \equiv 1 \!\!\!\! \pmod{p^n} \iff a \in \{1, p^n-1\}.$$
        \item[(ii)] If $p=2$ and $n \geq 3$, then
        $$a^2 \equiv 1 \!\!\!\! \pmod{p^n} \iff a \in \{1, 2^{n-1}-1, 2^{n-1}+1, 2^n-1\}.$$
    \end{itemize}
\end{lemma}
\begin{proof}
    Let $a \in \mathbb{Z}_{p^n}$.
    \begin{itemize}
        \item[(i)] $(\Leftarrow)$ If $a = 1$, then $a^2 \equiv 1 \pmod{p^n}$. On the other hand, if $a = p^n-1$, then $a \equiv (-1) \pmod{p^n}$, so $a^2 \equiv 1 \pmod{p^n}$.\\
        $(\Rightarrow)$ We have
        \vspace{-0.3cm}
        \begin{align*}
            & a^2 \equiv 1 \!\!\!\! \pmod{p^n}\\
            & \iff a^2-1 \equiv 0 \!\!\!\! \pmod{p^n}\\
            & \iff (a-1)(a+1) \equiv 0 \!\!\!\! \pmod{p^n}.
        \end{align*}
        
        So, we get $p^n \mid (a-1) \text{ or } p^n \mid (a+1)$.
        It means $a-1  \equiv 0 \pmod{p^n}$ or $a+1 \equiv 0  \pmod{p^n}$. Hence $$a \equiv 1 \!\!\!\! \pmod{p^n} \text{ or } a \equiv (-1) \!\!\!\! \pmod{p^n} \equiv p^n-1 \!\!\!\! \pmod{p^n}.$$
        As a result $a \in \{1, p^n-1\}$.
        \item[(ii)]  $(\Leftarrow)$ If $a=1$, then $a^2 \equiv 1 \pmod{2^n}$. Suppose $a=2^{n-1} \pm 1$, we get
        \begin{align*}
            a^2 & \equiv (2^{n-1} \pm 1)^2 \!\!\!\!\pmod{p^n}\\
             & \equiv \left((2^n)(2^{n-2}) \pm 2^n +1 \right)\!\!\!\! \pmod{p^n}\\
             & \equiv 1 \!\!\!\! \pmod{2^n}.
        \end{align*}
        Suppose $a = 2^n-1$, then $a \equiv (-1) \pmod{2^n}$. Hence $a^2 \equiv 1 \pmod{2^n}$.\\
        $(\Rightarrow)$ Analogously to the previous point, we obtain  
        $$(a-1)(a+1) \equiv 0 \pmod{2^n}.$$
        Consequently, it follows that $ 2^c \mid a+1$ and $2^{n-c} \mid a-1$ for $0 \leq c \leq n$. Assuming $2 \leq c \leq n-2$, we can write $a+1 = 2^c t$ and $a-1 = 2^{n-c} k$, where \( t, k \in \mathbb{Z}^+ \). So, we have  
        $$2^c t - 2^{n-c} k = 2 \iff 2\left(2^{c-2} t - 2^{n-c-2} k \right) = 1.$$
        This is impossible since $t$ and $k$ are integers.
        Thus, the possible values of $c$ that satisfy the condition are $\{0,1,n-1,n\}$. Consequently, there are several possible cases: $2^0 \mid a+1$ and $2^n \mid a-1$, or $2^1 \mid a+1$ and $2^{n-1} \mid a-1$, or $2^{n-1} \mid a+1$ and $2^1 \mid a-1$, or $2^n \mid a+1$ and $2^0 \mid a-1$. Therefore, the possible values of $a$ are  
        $$\{1, 2^{n-1}-1, 2^{n-1}+1, 2^n -1\}.$$
    \end{itemize} \end{proof}

\begin{theorem}\label{teo_strgrafZpn}
    Given ring $\mathbb{Z}_{p^n}$ with a prime number $p$ and a natural number $n$. It holds that
    \begin{align*}
        Cl_2(\mathbb{Z}_{p^n})=\begin{cases}
            K_1, &\text{ if } p=2, n=1,\\
            2K_1, &\text{ if } p=2, n=2,\\
            4K_1 \cup \left({2^{n-1}-2^{n-2}}-2 \right) K_2, &\text{ if } p=2, n \geq 3,\\
            2K_1 \cup \left(\frac{p^n-p^{n-1}}{2}-1 \right) K_2, &\text{ if } p \neq 2, n \geq 1.
        \end{cases}
    \end{align*}
\end{theorem}

\begin{proof}
    Suppose $p=2$ and $n=1$, we have $\mathbb{Z}_{p^n}=\mathbb{Z}_2$, then $Cl_2(\mathbb{Z}_2)=K_1$. Suppose $p=2$ and $n=2$, we have $Cl_2(\mathbb{Z}_{p^n})=Cl_2(\mathbb{Z}_4)=2K_1$.
    For any prime number and any natural number not included in the cases above, the following holds:
    \begin{align*}
    V(Cl_2(\mathbb{Z}_{p^n}))&=\{(1,u): u \in U(\mathbb{Z}_{p^n})\}\\
    &=\{(1,u): u \in \mathbb{Z}_{p^n} \setminus \langle p \rangle\}.
    \end{align*}
    Using Lemma \ref{lemma_U'U''(R)}, we get
    \begin{align*}
        |U'(\mathbb{Z}_{p^n})|=\begin{cases}
            2, &\text{ if } p \neq 2,\\
            4, &\text{ if } p = 2 \text{ and } n \geq 3.
        \end{cases}
    \end{align*}
    Since $|U''(\mathbb{Z}_{p^n})|=|U(\mathbb{Z}_{p^n})|-|U'(\mathbb{Z}_{p^n})|$, we derive
    \begin{align*}
        |U''(\mathbb{Z}_{p^n})|=\begin{cases}
            p^n-p^{n-1}-2, &\text{ if } p \neq 2,\\
            2^n-2^{n-1}-4, &\text{ if } p = 2 \text{ and } n \geq 3.
        \end{cases}
    \end{align*}
    Consequently, we obtain
    \begin{align*}
        Cl_2(\mathbb{Z}_{p^n})&=\begin{cases}
            2 K_1 \cup \frac{p^n-p^{n-1}-2}{2} K_2, &\text{ if } p \neq 2,\\
            4 K_1 \cup \frac{2^n-2^{n-1}-4}{2} K_2, &\text{ if } p = 2 \text{ and } n \geq 3
        \end{cases}\\
        Cl_2(\mathbb{Z}_{p^n})&=\begin{cases}
            2K_1 \cup \left(\frac{p^n-p^{n-1}}{2}-1 \right) K_2, &\text{ if } p \neq 2,\\
            4K_1 \cup \left({2^{n-1}-2^{n-2}}-2 \right) K_2, &\text{ jika } p = 2 \text{ and } n \geq 3.
        \end{cases}\\
    \end{align*}
\end{proof}

We now turn to the clean graph over the ring \( \mathbb{Z}_n \) with \( n \neq p^m \) for any prime number \( p \) and natural number \( m \). In this discussion, the structure of the clean graph over \( \mathbb{Z}_{p^nq^m} \) is presented for any two distinct prime numbers \( p,q \) and natural numbers \(n,m\). To obtain the structure of this graph, the Shuriken graph is first defined as follows.

\begin{definition}
    Given $t=2^k$, where $k \in \mathbb{Z}^+ \cup \{0\}$ and $n$ is a multiple of $t$. Shuriken graph $Sh^t_n$ is a graph that is isomorphic to the graph $G=(V(G),E(G))$ where
    $$V(G) = \{a_i, b_i, c_i : i = 1, 2, \dots, n\}$$
    and
    \begin{align*}  
		E(G) = &\{a_i b_j : i, j \in \{1, 2, \dots, n\}\} \text{ }\cup  
		\{a_i c_i, b_i c_i: i=1,2,\dots,t\} \text{ } \cup \\ & \{a_i c_{n+t+1-i}, b_i c_{n+t+1-i}: t+1 \leq i \leq n\} \text{ } \cup \\&\left\{a_i a_{n+t+1-i}, b_i b_{n+t+1-i}, c_i c_{n+t+1-i}: t+1 \leq i \leq \frac{n+t}{2}\right\}.
	\end{align*}  
\end{definition}

Examples of Shuriken graphs $Sh^2_6$, $Sh^4_{12}$, and $Sh^8_{16}$ are presented in Figures \ref{grafSh4} and \ref{grafSh8}.

\begin{figure}[H]
		\begin{center} 
			\resizebox{0.43\textwidth}{!}{\begin{tikzpicture}  
			[scale=.9,auto=center,roundnode/.style={circle,fill=blue!40}]
			\node[roundnode] (a1) at (39.22,23.54) {$a_1$};  
			\node[roundnode] (a2) at (45.73, 23.47)  {$b_1$};  
			\node[roundnode] (a3) at (42.63, 29.5)  {$c_1$};
			\node[roundnode] (a4) at (31.15, 15.96) {$a_2$};
			\node[roundnode] (a5) at (34.31, 20.74)  {$b_2$};
			\node[roundnode] (a6) at (27.8, 21.67) {$c_5$}; 
			\node[roundnode] (a7) at (33.56, 3.82) {$a_3$};  
			\node[roundnode] (a8) at (30.68, 9.5)  {$b_3$};  
			\node[roundnode] (a9) at (27.77,4.34)  {$c_4$};
			\node[roundnode] (a10) at (51.07, 20.3) {$a_4$};
			\node[roundnode] (a11) at (54.03, 15.18)  {$b_4$};
			\node[roundnode] (a12) at (57.43, 20.62) {$c_3$}; 
			\node[roundnode] (a13) at (54.04, 8.68) {$a_5$};  
			\node[roundnode] (a14) at (51.05, 3.49)  {$b_5$};  
			\node[roundnode] (a15) at (57.8, 4.09)  {$c_2$};
			\node[roundnode] (a16) at (46.19, 0.47) {$a_6$};
			\node[roundnode] (a17) at (38.96, 0.4)  {$b_6$};
			\node[roundnode] (a18) at (42.56, -5.09) {$c_6$};
	
			\draw (a1) -- (a2);  
			\draw (a1) -- (a5);  
			\draw (a1) -- (a8);
			\draw (a1) -- (a11); 
			\draw (a1) -- (a14);
			\draw (a1) -- (a17);
			\draw (a4) -- (a2);  
			\draw (a4) -- (a5);  
			\draw (a4) -- (a8);
			\draw (a4) -- (a11); 
			\draw (a4) -- (a14);
			\draw (a4) -- (a17);
			\draw (a7) -- (a2);  
			\draw (a7) -- (a5);  
			\draw (a7) -- (a8);
			\draw (a7) -- (a11);
			\draw (a7) -- (a14);
			\draw (a7) -- (a17);
			\draw (a10) -- (a2);  
			\draw (a10) -- (a5);  
			\draw (a10) -- (a8);
			\draw (a10) -- (a11);
			\draw (a10) -- (a14);
			\draw (a10) -- (a17);
			\draw (a13) -- (a2);  
			\draw (a13) -- (a5);  
			\draw (a13) -- (a8);
			\draw (a13) -- (a11);
			\draw (a13) -- (a14);
			\draw (a13) -- (a17);
			\draw (a16) -- (a2);  
			\draw (a16) -- (a5);  
			\draw (a16) -- (a8);
			\draw (a16) -- (a11);
			\draw (a16) -- (a14);
			\draw (a16) -- (a17);
			\draw (a1) -- (a3);  
			\draw (a2) -- (a3);  
			\draw (a4) -- (a6);
			\draw (a5) -- (a6);
			\draw (a7) -- (a9);  
			\draw (a8) -- (a9);  
			\draw (a10) -- (a12);
			\draw (a11) -- (a12);
			\draw (a13) -- (a15);
			\draw (a14) -- (a15);
			\draw (a16) -- (a18);
			\draw (a17) -- (a18);
			\draw (a4) -- (a13);  
			\draw (a5) -- (a14);  
			\draw (a6) -- (a15);
			\draw (a7) -- (a10);  
			\draw (a8) -- (a11);  
			\draw (a9) -- (a12);
		\end{tikzpicture}}\resizebox{0.55\textwidth}{!}{\begin{tikzpicture}  
				[scale=.9,auto=center,roundnode/.style={circle,fill=blue!40}]
				\node[roundnode] (a1) at (-0.87, 5.67) {$a_1$};  
				\node[roundnode] (a2) at (0.75, 5.67)  {$b_1$};  
				\node[roundnode] (a3) at (-0.01, 7.33)  {$c_1$};
				\node[roundnode] (a4) at (5.67, 0.85) {$a_2$};
				\node[roundnode] (a5) at (5.67, -0.92)  {$b_2$};
				\node[roundnode] (a6) at (7.39, 0) {$c_2$}; 
                \node[roundnode] (a7) at (0.85, -5.67) {$a_3$};  
				\node[roundnode] (a8) at (-0.65, -5.67)  {$b_3$};  
				\node[roundnode] (a9) at (0, -7.32)  {$c_3$};
				\node[roundnode] (a10) at (-5.67, -0.94) {$a_4$};
				\node[roundnode] (a11) at (-5.67, 0.75)  {$b_4$};
				\node[roundnode] (a12) at (-7.38, 0) {$c_4$}; 
                \node[roundnode] (a13) at (2.15, 5.32) {$a_5$};  
				\node[roundnode] (a14) at (3.47, 4.57)  {$b_5$};  
				\node[roundnode] (a15) at (3.58, 6.44)  {$c_{12}$};
				\node[roundnode] (a16) at (4.48, 3.58) {$a_6$};
				\node[roundnode] (a17) at (5.29, 2.23)  {$b_6$};
				\node[roundnode] (a18) at (6.34, 3.92) {$c_{11}$}; 
                \node[roundnode] (a19) at (5.2, -2.43) {$a_7$};  
				\node[roundnode] (a20) at (4.49, -3.57)  {$b_7$};  
				\node[roundnode] (a21) at (6.28, -3.76)  {$c_{10}$};
				\node[roundnode] (a22) at (3.46, -4.57) {$a_8$};
				\node[roundnode] (a23) at (2.22, -5.29)  {$b_8$};
				\node[roundnode] (a24) at (3.68, -6.28) {$c_9$}; 
                \node[roundnode] (a25) at (-2.17, -5.31) {$a_{12}$};  
				\node[roundnode] (a26) at (-3.48, -4.56)  {$b_{12}$};  
				\node[roundnode] (a27) at (-3.56, -6.32)  {$c_5$};
				\node[roundnode] (a28) at (-4.48, -3.59) {$a_{11}$};
				\node[roundnode] (a29) at (-5.19, -2.45)  {$b_{11}$};
				\node[roundnode] (a30) at (-6.1, -3.8) {$c_6$}; 
                \node[roundnode] (a31) at (-5.25, 2.32) {$a_{10}$};  
				\node[roundnode] (a32) at (-4.45, 3.62)  {$b_{10}$};  
				\node[roundnode] (a33) at (-6.29, 3.86)  {$c_7$};
				\node[roundnode] (a34) at (-3.43, 4.6) {$a_{9}$};
				\node[roundnode] (a35) at (-2.2, 5.3)  {$b_{9}$};
				\node[roundnode] (a36) at (-3.62, 6.44) {$c_8$};

				\draw (a1) -- (a2);  
				\draw (a1) -- (a5);  
				\draw (a1) -- (a8);
				\draw (a1) -- (a11); 
                \draw (a1) -- (a14);  
				\draw (a1) -- (a17);  
				\draw (a1) -- (a20);
				\draw (a1) -- (a23); 
                \draw (a1) -- (a26);  
				\draw (a1) -- (a29);  
				\draw (a1) -- (a32);
				\draw (a1) -- (a35); 
                \draw (a4) -- (a2);  
				\draw (a4) -- (a5);  
				\draw (a4) -- (a8);
				\draw (a4) -- (a11); 
                \draw (a4) -- (a14);  
				\draw (a4) -- (a17);  
				\draw (a4) -- (a20);
				\draw (a4) -- (a23); 
                \draw (a4) -- (a26);  
				\draw (a4) -- (a29);  
				\draw (a4) -- (a32);
				\draw (a4) -- (a35); 
                \draw (a7) -- (a2);  
				\draw (a7) -- (a5);  
				\draw (a7) -- (a8);
				\draw (a7) -- (a11); 
                \draw (a7) -- (a14);  
				\draw (a7) -- (a17);  
				\draw (a7) -- (a20);
				\draw (a7) -- (a23); 
                \draw (a7) -- (a26);  
				\draw (a7) -- (a29);  
				\draw (a7) -- (a32);
				\draw (a7) -- (a35); 
                \draw (a10) -- (a2);  
				\draw (a10) -- (a5);  
				\draw (a10) -- (a8);
				\draw (a10) -- (a11); 
                \draw (a10) -- (a14);  
				\draw (a10) -- (a17);  
				\draw (a10) -- (a20);
				\draw (a10) -- (a23); 
                \draw (a10) -- (a26);  
				\draw (a10) -- (a29);  
				\draw (a10) -- (a32);
				\draw (a10) -- (a35);
                \draw (a13) -- (a2);  
				\draw (a13) -- (a5);  
				\draw (a13) -- (a8);
				\draw (a13) -- (a11); 
                \draw (a13) -- (a14);  
				\draw (a13) -- (a17);  
				\draw (a13) -- (a20);
				\draw (a13) -- (a23); 
                \draw (a13) -- (a26);  
				\draw (a13) -- (a29);  
				\draw (a13) -- (a32);
				\draw (a13) -- (a35); 
                \draw (a16) -- (a2);  
				\draw (a16) -- (a5);  
				\draw (a16) -- (a8);
				\draw (a16) -- (a11); 
                \draw (a16) -- (a14);  
				\draw (a16) -- (a17);  
				\draw (a16) -- (a20);
				\draw (a16) -- (a23); 
                \draw (a16) -- (a26);  
				\draw (a16) -- (a29);  
				\draw (a16) -- (a32);
				\draw (a16) -- (a35); 
                \draw (a19) -- (a2);  
				\draw (a19) -- (a5);  
				\draw (a19) -- (a8);
				\draw (a19) -- (a11); 
                \draw (a19) -- (a14);  
				\draw (a19) -- (a17);  
				\draw (a19) -- (a20);
				\draw (a19) -- (a23); 
                \draw (a19) -- (a26);  
				\draw (a19) -- (a29);  
				\draw (a19) -- (a32);
				\draw (a19) -- (a35); 
                \draw (a22) -- (a2);  
				\draw (a22) -- (a5);  
				\draw (a22) -- (a8);
				\draw (a22) -- (a11); 
                \draw (a22) -- (a14);  
				\draw (a22) -- (a17);  
				\draw (a22) -- (a20);
				\draw (a22) -- (a23); 
                \draw (a22) -- (a26);  
				\draw (a22) -- (a29);  
				\draw (a22) -- (a32);
				\draw (a22) -- (a35); 
                \draw (a25) -- (a2);  
				\draw (a25) -- (a5);  
				\draw (a25) -- (a8);
				\draw (a25) -- (a11); 
                \draw (a25) -- (a14);  
				\draw (a25) -- (a17);  
				\draw (a25) -- (a20);
				\draw (a25) -- (a23); 
                \draw (a25) -- (a26);  
				\draw (a25) -- (a29);  
				\draw (a25) -- (a32);
				\draw (a25) -- (a35); 
                \draw (a28) -- (a2);  
				\draw (a28) -- (a5);  
				\draw (a28) -- (a8);
				\draw (a28) -- (a11); 
                \draw (a28) -- (a14);  
				\draw (a28) -- (a17);  
				\draw (a28) -- (a20);
				\draw (a28) -- (a23); 
                \draw (a28) -- (a26);  
				\draw (a28) -- (a29);  
				\draw (a28) -- (a32);
				\draw (a28) -- (a35); 
                \draw (a31) -- (a2);  
				\draw (a31) -- (a5);  
				\draw (a31) -- (a8);
				\draw (a31) -- (a11); 
                \draw (a31) -- (a14);  
				\draw (a31) -- (a17);  
				\draw (a31) -- (a20);
				\draw (a31) -- (a23); 
                \draw (a31) -- (a26);  
				\draw (a31) -- (a29);  
				\draw (a31) -- (a32);
				\draw (a31) -- (a35); 
                \draw (a34) -- (a2);  
				\draw (a34) -- (a5);  
				\draw (a34) -- (a8);
				\draw (a34) -- (a11); 
                \draw (a34) -- (a14);  
				\draw (a34) -- (a17);  
				\draw (a34) -- (a20);
				\draw (a34) -- (a23); 
                \draw (a34) -- (a26);  
				\draw (a34) -- (a29);  
				\draw (a34) -- (a32);
				\draw (a34) -- (a35); 
                \draw (a1) -- (a3);  
				\draw (a2) -- (a3);  
				\draw (a4) -- (a6);
				\draw (a5) -- (a6);
                \draw (a7) -- (a9);  
				\draw (a8) -- (a9);  
				\draw (a10) -- (a12);
				\draw (a11) -- (a12);
                \draw (a13) -- (a15);  
				\draw (a14) -- (a15);  
				\draw (a16) -- (a18);
				\draw (a17) -- (a18);
                \draw (a19) -- (a21);  
				\draw (a20) -- (a21);  
				\draw (a22) -- (a24);
				\draw (a23) -- (a24);
                \draw (a25) -- (a27);  
				\draw (a26) -- (a27);  
				\draw (a28) -- (a30);
				\draw (a29) -- (a30);
                \draw (a31) -- (a33);  
				\draw (a32) -- (a33);  
				\draw (a34) -- (a36);
				\draw (a35) -- (a36);
                \draw (a13) -- (a25);  
				\draw (a14) -- (a26);  
				\draw (a15) -- (a27);
                \draw (a16) -- (a28);  
				\draw (a17) -- (a29);  
				\draw (a18) -- (a30);
                \draw (a19) -- (a31);  
				\draw (a20) -- (a32);  
				\draw (a21) -- (a33);
                \draw (a22) -- (a34);  
				\draw (a23) -- (a35);  
				\draw (a24) -- (a36);
			\end{tikzpicture}}\caption{Graphs $Sh^2_6$ and $Sh^4_{12}$}\label{grafSh4}
		\end{center}
\end{figure}
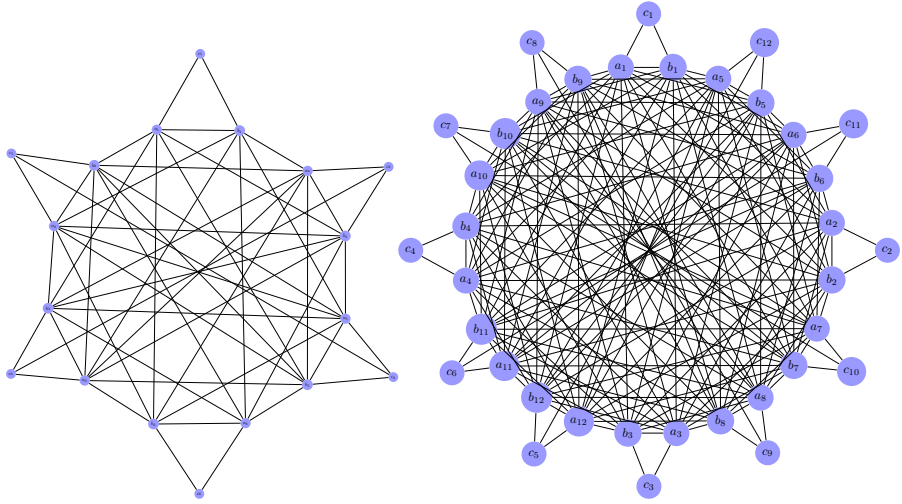
\begin{figure}[H]
    \begin{center}
            \resizebox{0.6\textwidth}{!}{\begin{tikzpicture}  
				[scale=.9,auto=center,roundnode/.style={circle,fill=blue!40}]
				\node[roundnode] (a1) at (-0.99, 11.89) {$a_1$};  
				\node[roundnode] (a2) at (0.99, 11.89)  {$b_1$};  
				\node[roundnode] (a3) at (0, 13.67)  {$c_1$};
				\node[roundnode] (a4) at (7.69, 9.16) {$a_2$};
				\node[roundnode] (a5) at (9.07, 7.75)  {$b_2$};
				\node[roundnode] (a6) at (9.7, 9.76) {$c_2$}; 
               \node[roundnode] (a7) at (11.89, 0.99) {$a_3$};
				\node[roundnode] (a8) at (11.89, -0.99)  {$b_3$};
				\node[roundnode] (a9) at (13.67, -0.01) {$c_3$}; 
				\node[roundnode] (a10) at (9.1, -7.72) {$a_4$};
				\node[roundnode] (a11) at (7.74, -9.08)  {$b_4$};
				\node[roundnode] (a12) at (9.59, -9.58) {$c_4$}; 
                \node[roundnode] (a13) at (0.99, -11.87) {$a_5$};  
				\node[roundnode] (a14) at (-1, -11.86)  {$b_5$};  
				\node[roundnode] (a15) at (0, -13.65)  {$c_{5}$};
				\node[roundnode] (a16) at (-7.66, -9.14) {$a_6$};
				\node[roundnode] (a17) at (-9.16, -7.64)  {$b_6$};
				\node[roundnode] (a18) at (-9.78, -9.7) {$c_{6}$}; 
                \node[roundnode] (a19) at (-11.89, -0.99) {$a_7$};  
				\node[roundnode] (a20) at (-11.89, 0.99)  {$b_7$};  
				\node[roundnode] (a21) at (-13.67, 0)  {$c_{7}$};
				\node[roundnode] (a22) at (-9.15, 7.65) {$a_8$};
				\node[roundnode] (a23) at (-7.64, 9.16)  {$b_8$};
				\node[roundnode] (a24) at (-9.7, 9.73) {$c_8$}; 
                \node[roundnode] (a25) at (3.72, 11.33) {$a_{9}$};  
				\node[roundnode] (a26) at (5.63, 10.52)  {$b_{9}$};  
				\node[roundnode] (a27) at (5.41, 12.6)  {$c_{16}$};
				\node[roundnode] (a28) at (10.43, 5.8) {$a_{10}$};
				\node[roundnode] (a29) at (11.31, 3.8)  {$b_{10}$};
				\node[roundnode] (a30) at (12.65, 5.67) {$c_{15}$}; 
                \node[roundnode] (a31) at (11.39, -3.54) {$a_{11}$};  
				\node[roundnode] (a32) at (10.66, -5.35)  {$b_{11}$};  
				\node[roundnode] (a33) at (12.57, -5.04)  {$c_{14}$};
				\node[roundnode] (a34) at (5.78, -10.44) {$a_{12}$};
				\node[roundnode] (a35) at (3.72, -11.33)  {$b_{12}$};
				\node[roundnode] (a36) at (5.5, -12.68) {$c_{13}$}; 
                \node[roundnode] (a37) at (-3.66, -11.35) {$a_{16}$};  
				\node[roundnode] (a38) at (-5.71, -10.47)  {$b_{16}$};  
				\node[roundnode] (a39) at (-5.44, -12.51)  {$c_{9}$};
				\node[roundnode] (a40) at (-10.6, -5.48) {$a_{15}$};
				\node[roundnode] (a41) at (-11.39, -3.54)  {$b_{15}$};
				\node[roundnode] (a42) at (-12.74, -5.24) {$c_{10}$}; 
                \node[roundnode] (a43) at (-11.37, 3.59) {$a_{14}$};  
				\node[roundnode] (a44) at (-10.59, 5.49)  {$b_{14}$};  
				\node[roundnode] (a45) at (-12.59, 5.27)  {$c_{11}$};
				\node[roundnode] (a46) at (-5.58, 10.54) {$a_{13}$};
				\node[roundnode] (a47) at (-3.61, 11.37)  {$b_{13}$};
				\node[roundnode] (a48) at (-5.32, 12.65) {$c_{12}$};

				\draw (a1) -- (a2);  
				\draw (a1) -- (a5);  
				\draw (a1) -- (a8);
				\draw (a1) -- (a11); 
                \draw (a1) -- (a14);  
				\draw (a1) -- (a17);  
				\draw (a1) -- (a20);
				\draw (a1) -- (a23); 
                \draw (a1) -- (a26);  
				\draw (a1) -- (a29);  
				\draw (a1) -- (a32);
				\draw (a1) -- (a35); 
                \draw (a1) -- (a38);
				\draw (a1) -- (a41); 
                \draw (a1) -- (a44);  
				\draw (a1) -- (a47);  
                \draw (a4) -- (a2);  
				\draw (a4) -- (a5);  
				\draw (a4) -- (a8);
				\draw (a4) -- (a11); 
                \draw (a4) -- (a14);  
				\draw (a4) -- (a17);  
				\draw (a4) -- (a20);
				\draw (a4) -- (a23); 
                \draw (a4) -- (a26);  
				\draw (a4) -- (a29);  
				\draw (a4) -- (a32);
				\draw (a4) -- (a35); 
                \draw (a4) -- (a38);  
				\draw (a4) -- (a41);  
				\draw (a4) -- (a44);
				\draw (a4) -- (a47); 
                \draw (a7) -- (a2);  
				\draw (a7) -- (a5);  
				\draw (a7) -- (a8);
				\draw (a7) -- (a11); 
                \draw (a7) -- (a14);  
				\draw (a7) -- (a17);  
				\draw (a7) -- (a20);
				\draw (a7) -- (a23); 
                \draw (a7) -- (a26);  
				\draw (a7) -- (a29);  
				\draw (a7) -- (a32);
				\draw (a7) -- (a35); 
                \draw (a7) -- (a38);  
				\draw (a7) -- (a41);  
				\draw (a7) -- (a44);
				\draw (a7) -- (a47); 
                \draw (a10) -- (a2);  
				\draw (a10) -- (a5);  
				\draw (a10) -- (a8);
				\draw (a10) -- (a11); 
                \draw (a10) -- (a14);  
				\draw (a10) -- (a17);  
				\draw (a10) -- (a20);
				\draw (a10) -- (a23); 
                \draw (a10) -- (a26);  
				\draw (a10) -- (a29);  
				\draw (a10) -- (a32);
				\draw (a10) -- (a35);
                \draw (a10) -- (a38);  
				\draw (a10) -- (a41);  
				\draw (a10) -- (a44);
				\draw (a10) -- (a47);
                \draw (a13) -- (a2);  
				\draw (a13) -- (a5);  
				\draw (a13) -- (a8);
				\draw (a13) -- (a11); 
                \draw (a13) -- (a14);  
				\draw (a13) -- (a17);  
				\draw (a13) -- (a20);
				\draw (a13) -- (a23); 
                \draw (a13) -- (a26);  
				\draw (a13) -- (a29);  
				\draw (a13) -- (a32);
				\draw (a13) -- (a35); 
                \draw (a13) -- (a38);  
				\draw (a13) -- (a41);  
				\draw (a13) -- (a44);
				\draw (a13) -- (a47); 
                \draw (a16) -- (a2);  
				\draw (a16) -- (a5);  
				\draw (a16) -- (a8);
				\draw (a16) -- (a11); 
                \draw (a16) -- (a14);  
				\draw (a16) -- (a17);  
				\draw (a16) -- (a20);
				\draw (a16) -- (a23); 
                \draw (a16) -- (a26);  
				\draw (a16) -- (a29);  
				\draw (a16) -- (a32);
				\draw (a16) -- (a35); 
                \draw (a16) -- (a38);  
				\draw (a16) -- (a41);  
				\draw (a16) -- (a44);
				\draw (a16) -- (a47); 
                \draw (a19) -- (a2);  
				\draw (a19) -- (a5);  
				\draw (a19) -- (a8);
				\draw (a19) -- (a11); 
                \draw (a19) -- (a14);  
				\draw (a19) -- (a17);  
				\draw (a19) -- (a20);
				\draw (a19) -- (a23); 
                \draw (a19) -- (a26);  
				\draw (a19) -- (a29);  
				\draw (a19) -- (a32);
				\draw (a19) -- (a35); 
                \draw (a19) -- (a38);  
				\draw (a19) -- (a41);  
				\draw (a19) -- (a44);
				\draw (a19) -- (a47); 
                \draw (a22) -- (a2);  
				\draw (a22) -- (a5);  
				\draw (a22) -- (a8);
				\draw (a22) -- (a11); 
                \draw (a22) -- (a14);  
				\draw (a22) -- (a17);  
				\draw (a22) -- (a20);
				\draw (a22) -- (a23); 
                \draw (a22) -- (a26);  
				\draw (a22) -- (a29);  
				\draw (a22) -- (a32);
				\draw (a22) -- (a35); 
                \draw (a22) -- (a38);  
				\draw (a22) -- (a41);  
				\draw (a22) -- (a44);
				\draw (a22) -- (a47); 
                \draw (a25) -- (a2);  
				\draw (a25) -- (a5);  
				\draw (a25) -- (a8);
				\draw (a25) -- (a11); 
                \draw (a25) -- (a14);  
				\draw (a25) -- (a17);  
				\draw (a25) -- (a20);
				\draw (a25) -- (a23); 
                \draw (a25) -- (a26);  
				\draw (a25) -- (a29);  
				\draw (a25) -- (a32);
				\draw (a25) -- (a35); 
                \draw (a25) -- (a38);  
				\draw (a25) -- (a41);  
				\draw (a25) -- (a44);
				\draw (a25) -- (a47); 
                \draw (a28) -- (a2);  
				\draw (a28) -- (a5);  
				\draw (a28) -- (a8);
				\draw (a28) -- (a11); 
                \draw (a28) -- (a14);  
				\draw (a28) -- (a17);  
				\draw (a28) -- (a20);
				\draw (a28) -- (a23); 
                \draw (a28) -- (a26);  
				\draw (a28) -- (a29);  
				\draw (a28) -- (a32);
				\draw (a28) -- (a35); 
                \draw (a28) -- (a38);  
				\draw (a28) -- (a41);  
				\draw (a28) -- (a44);
				\draw (a28) -- (a47); 
                \draw (a31) -- (a2);  
				\draw (a31) -- (a5);  
				\draw (a31) -- (a8);
				\draw (a31) -- (a11); 
                \draw (a31) -- (a14);  
				\draw (a31) -- (a17);  
				\draw (a31) -- (a20);
				\draw (a31) -- (a23); 
                \draw (a31) -- (a26);  
				\draw (a31) -- (a29);  
				\draw (a31) -- (a32);
				\draw (a31) -- (a35); 
                \draw (a31) -- (a38);  
				\draw (a31) -- (a41);  
				\draw (a31) -- (a44);
				\draw (a31) -- (a47); 
                \draw (a34) -- (a2);  
				\draw (a34) -- (a5);  
				\draw (a34) -- (a8);
				\draw (a34) -- (a11); 
                \draw (a34) -- (a14);  
				\draw (a34) -- (a17);  
				\draw (a34) -- (a20);
				\draw (a34) -- (a23); 
                \draw (a34) -- (a26);  
				\draw (a34) -- (a29);  
				\draw (a34) -- (a32);
				\draw (a34) -- (a35); 
                \draw (a34) -- (a38);  
				\draw (a34) -- (a41);  
				\draw (a34) -- (a44);
				\draw (a34) -- (a47); 
                \draw (a37) -- (a2);  
				\draw (a37) -- (a5);  
				\draw (a37) -- (a8);
				\draw (a37) -- (a11); 
                \draw (a37) -- (a14);  
				\draw (a37) -- (a17);  
				\draw (a37) -- (a20);
				\draw (a37) -- (a23); 
                \draw (a37) -- (a26);  
				\draw (a37) -- (a29);  
				\draw (a37) -- (a32);
				\draw (a37) -- (a35); 
                \draw (a37) -- (a38);  
				\draw (a37) -- (a41);  
				\draw (a37) -- (a44);
				\draw (a37) -- (a47);
                \draw (a40) -- (a2);  
				\draw (a40) -- (a5);  
				\draw (a40) -- (a8);
				\draw (a40) -- (a11); 
                \draw (a40) -- (a14);  
				\draw (a40) -- (a17);  
				\draw (a40) -- (a20);
				\draw (a40) -- (a23); 
                \draw (a40) -- (a26);  
				\draw (a40) -- (a29);  
				\draw (a40) -- (a32);
				\draw (a40) -- (a35); 
                \draw (a40) -- (a38);  
				\draw (a40) -- (a41);  
				\draw (a40) -- (a44);
				\draw (a40) -- (a47); 
                \draw (a43) -- (a2);  
				\draw (a43) -- (a5);  
				\draw (a43) -- (a8);
				\draw (a43) -- (a11); 
                \draw (a43) -- (a14);  
				\draw (a43) -- (a17);  
				\draw (a43) -- (a20);
				\draw (a43) -- (a23); 
                \draw (a43) -- (a26);  
				\draw (a43) -- (a29);  
				\draw (a43) -- (a32);
				\draw (a43) -- (a35); 
                \draw (a43) -- (a38);  
				\draw (a43) -- (a41);  
				\draw (a43) -- (a44);
				\draw (a43) -- (a47); 
                \draw (a46) -- (a2);  
				\draw (a46) -- (a5);  
				\draw (a46) -- (a8);
				\draw (a46) -- (a11); 
                \draw (a46) -- (a14);  
				\draw (a46) -- (a17);  
				\draw (a46) -- (a20);
				\draw (a46) -- (a23); 
                \draw (a46) -- (a26);  
				\draw (a46) -- (a29);  
				\draw (a46) -- (a32);
				\draw (a46) -- (a35); 
                \draw (a46) -- (a38);  
				\draw (a46) -- (a41);  
				\draw (a46) -- (a44);
				\draw (a46) -- (a47); 
                \draw (a1) -- (a3);  
				\draw (a2) -- (a3);  
				\draw (a4) -- (a6);
				\draw (a5) -- (a6);
                \draw (a7) -- (a9);  
				\draw (a8) -- (a9);  
				\draw (a10) -- (a12);
				\draw (a11) -- (a12);
                \draw (a13) -- (a15);  
				\draw (a14) -- (a15);  
				\draw (a16) -- (a18);
				\draw (a17) -- (a18);
                \draw (a19) -- (a21);  
				\draw (a20) -- (a21);  
				\draw (a22) -- (a24);
				\draw (a23) -- (a24);
                \draw (a25) -- (a27);  
				\draw (a26) -- (a27);  
				\draw (a28) -- (a30);
				\draw (a29) -- (a30);
                \draw (a31) -- (a33);  
				\draw (a32) -- (a33);  
				\draw (a34) -- (a36);
				\draw (a35) -- (a36);
                \draw (a37) -- (a39);
				\draw (a38) -- (a39);
                \draw (a40) -- (a42);  
				\draw (a41) -- (a42);  
				\draw (a43) -- (a45);
				\draw (a44) -- (a45);
                \draw (a46) -- (a48);
				\draw (a47) -- (a48);
                
                \draw (a25) -- (a37);  
				\draw (a26) -- (a38);  
				\draw (a27) -- (a39);
                \draw (a28) -- (a40);  
				\draw (a29) -- (a41);  
				\draw (a30) -- (a42);
                \draw (a31) -- (a43);  
				\draw (a32) -- (a44);  
				\draw (a33) -- (a45);
                \draw (a34) -- (a46);  
				\draw (a35) -- (a47);  
				\draw (a36) -- (a48);
			\end{tikzpicture}}\caption{Graph $Sh^8_{16}$}\label{grafSh8}
		\end{center}
\end{figure}
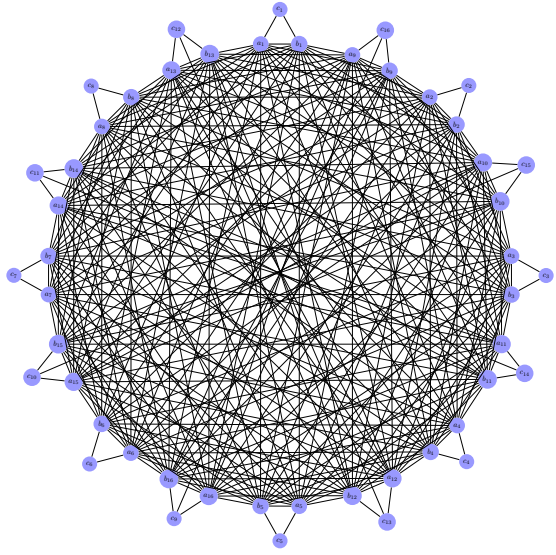


In the next step, we examine the structure of the clean graph $Cl_2(\mathbb{Z}_{p^nq^m})$ for any distinct prime numbers $p,q$ and natural numbers $n,m$. We know that $\mathbb{Z}_{p^nq^m} = \mathbb{Z}_{p^n} \times \mathbb{Z}_{q^m}$. It is observed that
	\begin{align*}
		V(Cl(\mathbb{Z}_{p^nq^m})) = & \left(Id(\mathbb{Z}_{p^n}) \times Id(\mathbb{Z}_{q^m})\right) \times \left(U(\mathbb{Z}_{p^n}) \times U(\mathbb{Z}_{q^m})\right) \\
		= & \{(0,0), (0,1), (1,0), (1,1)\} \\
        &\times \{(u_1,u_2): u_1 \in U(\mathbb{Z}_{p^n}), u_2 \in U(\mathbb{Z}_{q^m})\},
	\end{align*}
	then,
	$$V(Cl_2(\mathbb{Z}_{p^nq^m})) = \{(0,1), (1,0), (1,1)\} \times \{(u_1,u_2): u_1 \in U(\mathbb{Z}_{p^n}), u_2 \in U(\mathbb{Z}_{q^m})\}.$$
Furthermore, the structure of the graph $Cl_2(\mathbb{Z}_{p^nq^m})$ is obtained, as given in the following theorem.
\begin{theorem}\label{isomorfisCl2Zpnqn}
    For any distinct prime numbers $p,q$ and natural numbers $n,m$, the following holds
    $$Cl_2(\mathbb{Z}_{p^nq^m}) \cong Sh^{|U'(\mathbb{Z}_{p^nq^m})|}_{(p^n-p^{n-1})(q^m-q^{m-1})}.$$
\end{theorem}
\begin{proof} 
Assume that $t=|U'(\mathbb{Z}_{p^nq^m})|$ and $k=(p^n-p^{n-1})(q^m-q^{m-1})$. Let 
$$U'(\mathbb{Z}_{p^nq^m})=\{x_1,x_2,\dots,x_t\} \text{ and } U''(\mathbb{Z}_{p^nq^m})=\{x_{t+1},x_{t+2},\dots,x_{k}\},$$
such that $x_ix_{k+t+1-i}=(1,1)$, for any $i=t+1, t+2, \dots,k$. Bijective mapping $$g: V(Cl_2(\mathbb{Z}_{p^nq^m})) \to V(Sh^t_{k})$$ is defined by
        \begin{align*}
            g((0,1),x_i)=a_i,\\
            g((1,0),x_i)=b_i,\\
            g((1,1),x_i)=c_i.
        \end{align*}
for all $i = 1,2, \dots, k$. Let $(e,u), (f,v) \in V(Cl_2(\mathbb{Z}_{p^nq^m}))$ be arbitrary, with $(e,u)(f,v) \in E(Cl_2(\mathbb{Z}_{p^nq^m}))$. This implies that $ef = 0$ or $uv = 1$.
		\begin{enumerate}
			\item If $ef = 0$, then $e=(0,1)$ and $f=(1,0)$, or $e=(1,0)$ and $f=(0,1)$. So, $g(e,u)=a_i$ and $g(f,v)=b_j$, or we get $g(e,u)=b_i$ and $g(f,v)=a_j$ for some $i,j \in \{1,2,\dots,k\}$. Hence, $g(e,u)g(u,v) \in E(Sh^t_k)$.
			\item Suppose $uv = 1$. Since $u, v \in \{x_i: i = 1, 2, \dots, k\}$, the following possibilities occur:
            \begin{enumerate}
                \item If $u = v = x_i$ with $i \in \{1,2,\dots,t\}$ and $e \neq f$, then $$g(e,u),g(f,v) \in \{a_i,b_i,c_i\}$$ where $g(e,u) \neq g(f,v)$. Hence $g(e,u)g(f,v) \in E(Sh^t_k)$. 
                \item If $u = x_i, v = x_{k+t+1-i}$ with $i \in \{t+1, t+2, \dots, k\}$, then $$g(e,u) \in \{a_i,b_i,c_i\} \text{ and } g(f,v) \in \{a_{k+t+1-i}, b_{k+t+1-i}, c_{k+t+1-i}\}.$$ Hence $g(e,u)g(f,v) \in E(Sh^t_k)$.
            \end{enumerate}
		\end{enumerate}
    From all the cases above, it follows that $g(e,u)g(f,v) \in E(Sh^t_k)$.  
    In contrast, consider an arbitrary $(e,u), (f,v) \in V(Cl_2(\mathbb{Z}_{p^nq^m}))$ such that $g(e,u)g(f,v) \in E(Sh^t_k)$. Based on the definition of the Shuriken graph, several possibilities arise as follows.
    \begin{enumerate}
        \item If $g(e,u)=a_i$ and $g(f,v)=b_j$ with $i,j \in \{1,2,\dots,k\}$, then $e=(0,1)$ and $f=(1,0)$, so $(e,u)(f,v) \in E(Cl_2(\mathbb{Z}_{p^nq^m}))$.
        \item If $g(e,u) \in \{a_i, b_i\}$ and $g(f,v)=c_i$ with $i \in \{1,2,\dots,t\}$, then $u=v=x_i$. Hence, $(e,u)(f,v) \in E(Cl_2(\mathbb{Z}_{p^nq^m}))$.
        \item If $g(e,u) \in \{a_i, b_i\}$ and $g(f,v)= c_{k+t+1-i}$ with $i \in \{t+1,t+2,\dots,k\}$, then $u=x_i$ and $v=x_{k+t+1-i}$, so $(e,u)(f,v) \in E(Cl_2(\mathbb{Z}_{p^nq^m}))$.
        \item If $\{g(e,u), g(f,v)\} \in \left\{\{a_i, a_{k+t+1-i}\}, \{b_i, b_{k+t+1-i}\}, \{c_i, c_{k+t+1-i}\}\right\}$ with $i \in \{t+1,t+2,\dots,\frac{k+t}{2}\}$, then $u=x_i$ and $v=x_{k+t+1-i}$. Hence $(e,u)(f,v) \in E(Cl_2(\mathbb{Z}_{p^nq^m}))$.
    \end{enumerate}
    From all the cases above, it follows that $(e,u)(f,v) \in E(Cl_2(\mathbb{Z}_{p^nq^m}))$. Hence, $g$ is a graph isomorphism $Cl_2(\mathbb{Z}_{p^nq^m})$ to $Sh^t_k$.
\end{proof}

\begin{proposition}\label{propisoZpnqm}
    For any distinct prime numbers $p,q$ and natural numbers $n,m$, the following holds
    $$Cl_2(\mathbb{Z}_{p^nq^m}) \cong \begin{cases}
        Sh^2_{q^m-q^{m-1}}, & \text{ if } p^n=2,\\
        Sh^4_{2(q^m-q^{m-1})}, & \text{ if } p^n=4,\\
        Sh^8_{(p^n-p^{n-1})(q^m-q^{m-1})}, & \text{ if } p=2, n \geq 3,\\
        Sh^4_{(p^n-p^{n-1})(q^m-q^{m-1})}, & \text{ if } p,q \neq 2.
    \end{cases}$$
\end{proposition}
\begin{proof}
    From Theorem \ref{isomorfisCl2Zpnqn}, we know that
    $$Cl_2(\mathbb{Z}_{p^nq^m}) \cong Sh^{|U'(\mathbb{Z}_{p^nq^m})|}_{(p^n-p^{n-1})(q^m-q^{m-1})}.$$
    Based on Lemma \ref{lemma_U'U''(R)}, the following four cases are considered.
    \begin{enumerate}
        \item Case $p^n=2$. So, $q \neq 2$. We have $|U(\mathbb{Z}_{p^nq^m})|=q^m-q^{m-1}$ and
        \begin{align*}
            \left(|U'(\mathbb{Z}_{p^n})|=1 \text{, }  |U'(\mathbb{Z}_{q^m})|=2 \right) \Rightarrow  |U'(\mathbb{Z}_{p^nq^m})|=2.
        \end{align*}
        
        \item Case $p^n=4$. So, $q \neq 2$. We have $|U(\mathbb{Z}_{p^nq^m})|=2(q^m-q^{m-1})$ and
        \begin{align*}
            \left(|U'(\mathbb{Z}_{p^n})|=2 \text{, }  |U'(\mathbb{Z}_{q^m})|=2\right) \Rightarrow |U'(\mathbb{Z}_{p^nq^m})|=4.
        \end{align*}
        
        \item Case $p=2$ and $n \geq 3$. So, $q \neq 2$. We have
        \begin{align*}
            |U(\mathbb{Z}_{p^nq^m})|&=(p^n-p^{n-1})(q^m-q^{m-1}) \text{ and }\\
            (|U'(\mathbb{Z}_{p^n})|&=4 \text{, } |U'(\mathbb{Z}_{q^m})|=2) \Rightarrow |U'(\mathbb{Z}_{p^nq^m})|=8.
        \end{align*}
       
        \item Case $p,q \neq 2$. We have $|U(\mathbb{Z}_{p^nq^m})|=(p^n-p^{n-1})(q^m-q^{m-1})$ and
        \begin{align*}
            \left(|U'(\mathbb{Z}_{p^n})|=2 \text{, } |U'(\mathbb{Z}_{q^m})|=2\right) \Rightarrow |U'(\mathbb{Z}_{p^nq^m})|=4.
        \end{align*}
    \end{enumerate}
\end{proof}

\subsection{Structure of Clean Graphs over $\mathbb{Z}_{\tiny{p_1^{n_1}p_2^{n_2}p_3^{n_3}}}$, $\mathbb{Z}_{\tiny{p_1^{n_1}p_2^{n_2}p_3^{n_3}p_4^{n_4}}}$, and Generalized Structure of Clean Graph over $\mathbb{Z}_n$}

Before discussing the structure of graphs over $\mathbb{Z}_n$ for any natural number $n$, we introduce a generalization of the shuriken graph, referred to as the shuriken operation on a graph, as defined below.

\begin{definition} Let $G=(V(G), E(G))$ be a graph and let $n, t$ be positive integers such that $n-t$ is even. The $(t,n)$-shuriken graph of $G$, denoted by $Shu^t_n(G)$, is constructed from $G$ by first adding a new vertex ${z}$ and then creating $n$ copies of the resulting graph. Let $G'_i$ for $1 \leq i \leq n$ denote the $i$-th copy of $G$ after the addition of the new vertex. The vertex set and edge set of the graph $Shu^t_n(G)$ are given by:
        $$V(Shu^t_n(G))=\bigcup_{i=1}^n \{z_i, v_i: v \in V(G)\}$$ and 
    \begin{align*}
        E(Shu^t_n(G))= &\{u_iv_j : uv \in E(G), i,j \in \{1,2,\dots,n\}\} \\
        &\cup \{u_iv_i: u_i,v_i \in V(G'_i), u_i\neq v_i, i \in\{1,2,\dots,t\}\}\\ 
        &\cup \Bigg\{u_iv_{n+t+1-i}: u_i \in V(G'_i), v_{n+t+1-i} \in V(G'_{n+t+1-i}) \\
        & \hspace{2.5 cm}i \in \left\{t+1,t+2,\dots,\frac{n+t}{2}\right\} \Bigg\}.
    \end{align*}
\end{definition}

Given graph $P_3$, with $V(P_3)=\{a,b,c\}$. The $(2,4)$-shuriken graph of $P_3$, $Shu^2_4(P_3)$ is presented in Figures \ref{grafShu24P3}.

\begin{figure}[H]
    \begin{center}
        \resizebox{0.55\textwidth}{!}{\begin{tikzpicture}  
				[scale=.9,auto=center,roundnode/.style={circle,fill=blue!40}]
				\node[roundnode] (a1) at (-2,6) {$a_1$};  
				\node[roundnode] (a2) at (0,6)  {$b_1$};  
				\node[roundnode] (a3) at (2,6)  {$c_1$};
				\node[roundnode] (a4) at (1,8) {$z_1$};
				\node[roundnode] (a5) at (2,-2)  {$a_2$};
				\node[roundnode] (a6) at (0,-2) {$b_2$}; 
                \node[roundnode] (a7) at (-2,-2) {$c_2$};  
				\node[roundnode] (a8) at (-1,-4)  {$z_2$};  
				\node[roundnode] (a9) at (4,4)  {$a_3$};
				\node[roundnode] (a10) at (4,2) {$b_3$};
				\node[roundnode] (a11) at (4,0)  {$c_3$};
				\node[roundnode] (a12) at (6,1) {$z_4$}; 
                \node[roundnode] (a13) at (-4,0) {$a_4$};  
				\node[roundnode] (a14) at (-4,2)  {$b_4$};  
				\node[roundnode] (a15) at (-4,4)  {$c_4$};
				\node[roundnode] (a16) at (-6,3) {$z_3$};

				\draw (a1) -- (a2);  
				\draw (a2) -- (a3);
                \draw (a1) -- (a4);  
				\draw (a2) -- (a4);  
				\draw (a3) -- (a4);%
                \draw[bend left=30] (a1) to (a3);
                \draw (a2) -- (a5);  
				\draw (a2) -- (a7);
                \draw (a2) -- (a9);  
				\draw (a2) -- (a11);
                \draw (a2) -- (a13);
                \draw (a2) -- (a15);
				\draw (a5) -- (a6);
				\draw (a6) -- (a7); 
                \draw (a5) -- (a8);  
				\draw (a6) -- (a8);  
				\draw (a7) -- (a8);%
                \draw[bend left=30] (a5) to (a7);
                \draw (a6) -- (a1);  
				\draw (a6) -- (a3);
                \draw (a6) -- (a9);  
				\draw (a6) -- (a11);
                \draw (a6) -- (a13);
                \draw (a6) -- (a15);
                \draw (a9) -- (a10);  
				\draw (a10) -- (a11); 
                \draw (a9) -- (a12);  
				\draw (a10) -- (a12);  
				\draw (a11) -- (a12);%
                \draw (a10) -- (a5);  
				\draw (a10) -- (a7);
                \draw (a10) -- (a1);  
				\draw (a10) -- (a3);
                \draw (a10) -- (a13);
                \draw (a10) -- (a15);
				\draw (a13) -- (a14);
				\draw (a14) -- (a15); 
                \draw (a13) -- (a16);  
				\draw (a14) -- (a16);  
				\draw (a15) -- (a16);%
                \draw (a14) -- (a5);  
				\draw (a14) -- (a7);
                \draw (a14) -- (a9);  
				\draw (a14) -- (a11);
                \draw (a14) -- (a1);
                \draw (a14) -- (a3);
                \draw (a9) -- (a13);
				\draw (a9) -- (a14); 
                \draw (a9) -- (a15);  
				\draw (a10) -- (a13);  
				\draw (a10) -- (a14);
                \draw (a10) -- (a15);  
				\draw (a11) -- (a13);
                \draw (a11) -- (a14);  
				\draw (a11) -- (a15);
                \draw (a12) -- (a16);
			\end{tikzpicture}}\caption{Graph $Shu^2_{4}(P_3)$}\label{grafShu24P3}
		\end{center}
\end{figure}
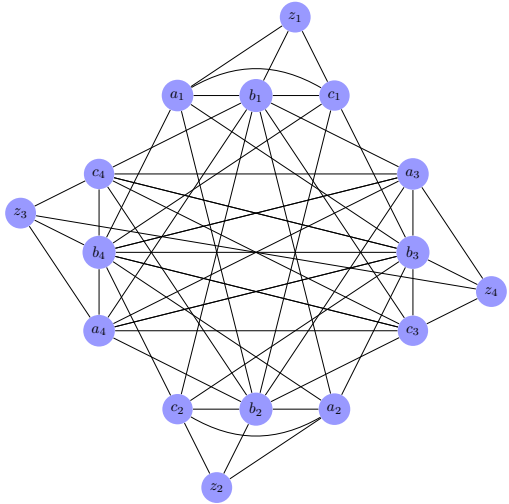

We now present a discussion on the properties of the $(t,n)$-shuriken graph for any given graph $G$, as presented in the following theorems.

\begin{theorem}
    Given any graph $G$ and positive integers $n \geq t \geq 2$ such that $n - t$ is even. The $(t, n)$-shuriken graph of $G$ is disconnected if and only if the graph $G$ is a null graph.
\end{theorem}
\begin{proof}
    $(\Longleftarrow)$ Given a null graph $G$. Let $|V(G)|=k$.  
        We have
        \begin{align*}
            E(Shu^t_n(G))= & \{u_iv_i: u_i,v_i \in V(Shu^t_n(G)), u_i\neq v_i, i \in\{1,2,\dots,n\}\}\\
            &\cup \Bigg\{u_iv_{n+t+1-i}: u_i,v_{n+t+1-i} \in V(Shu^t_n(G)) \\
        & \hspace{2.5 cm}i \in \left\{t+1,t+2,\dots,\frac{n+t}{2}\right\} \Bigg\}.
        \end{align*}

        There exists sets of vertices $X_1=\{z_1,v_1: v \in V(G)\}$ and $X_2=\{z_2,v_2: v \in V(G)\} \subseteq V(Shu^t_n(G))$, such that no two points in $X_1$ and $X_2$ are connected to each other. \\
    $(\Longrightarrow)$ Let $G$ be a graph that is not null. This means that there exists an edge $uv \in E(G)$. Let $x_i, y_j \in V(Shu^t_n(G))$ be arbitrary, with $i, j \in {1, 2, \dots, n}$. This implies that $x_i \in V(G'_i)$ and $y_j \in V(G'_j)$. The following cases are considered:
    \begin{enumerate}
        \item Case $x_i \neq u_i$ and $y_j \neq v_j$. Several paths are available as follows:
        \begin{align*}
            &x_i-u_i-v_j-y_j, & \text{ if } 1 \leq i,j \leq t,\\
            &x_i-u_i-v_{n+t+1-j}-y_j, & \text{ if } 1 \leq i \leq t \text{ and } t+1 \leq j \leq n,\\
            &x_i-u_{n+t+1-i}-v_j-y_j, & \text{ if } t+1 \leq i \leq n \text{ and } 1 \leq j \leq t,\\
            &x_i-u_{n+t+1-i}-v_{n+t+1-j}-y_j, & \text{ if } t+1 \leq i,j \leq n.
        \end{align*}
        \item Case $x_i = u_i$ and $y_j \neq v_j$. The following paths are obtained:
        \begin{align*}
            &x_i-v_j-y_j, & \text{ if } 1 \leq i,j \leq t,\\
            &x_i-v_{n+t+1-j}-y_j, & \text{ if } 1 \leq i \leq t \text{ and } t+1 \leq j \leq n,\\
            &x_i-v_j-y_j, & \text{ if } t+1 \leq i \leq n \text{ and } 1 \leq j \leq t,\\
            &x_i-v_{n+t+1-j}-y_j, & \text{ if } t+1 \leq i,j \leq n.
        \end{align*}
        \item Case $x_i \neq u_i$ and $y_j = v_j$. The following paths are obtained:
        \begin{align*}
            &x_i-u_i-y_j, & \text{ if } 1 \leq i,j \leq t,\\
            &x_i-u_i-y_j, & \text{ if } 1 \leq i \leq t \text{ and } t+1 \leq j \leq n,\\
            &x_i-u_{n+t+1-i}-y_j, & \text{ if } t+1 \leq i \leq n \text{ and } 1 \leq j \leq t,\\
            &x_i-u_{n+t+1-i}-y_j, & \text{ if } t+1 \leq i,j \leq n.
        \end{align*}
        \item If $x_i = u_i$ and $y_j = v_j$, then $x_iy_j \in E(Shu^t_n(G))$.
    \end{enumerate}
\end{proof}

\begin{proposition}\label{propShtnkeIR}
    Given any connected graphs $G_1$ and $G_2$, and positive integers $t$ and $n$, with $2 \leq t < n$ and $n - t$ even. The following holds:
    $$Shu^t_n(G_1) \cong Shu^t_n(G_2) \iff G_1 \cong G_2.$$
\end{proposition}
\begin{proof}
    Since $Shu^t_n(G_1) \cong Shu^t_n(G_2)$, we have $$|V(Shu^t_n(G))| = |V(Shu^t_n(H))| \text{ and } |E(Shu^t_n(G))| = |E(Shu^t_n(H))|.$$ So, $|V(G_1)|=|V(G_2)|=k$. Let $i,j \in\{t+1,t+2,\dots,n\}$ be arbitrary. Let $$X_1 = \{u_i: u \in V(G_1)\}$$ and $$X_2=\{v_j: v \in V(G_2)\}.$$ 
    The following bijective functions can be constructed. $$\psi: V(\langle X_1 \rangle _{Shu^t_n(G_1)}) \to V(G_1) \text{ and } \phi: V(\langle X_2 \rangle _{Shu^t_n(G_2})) \to V(G_2),$$ with $\psi(u_i)=u$ for each $u_i \in V(\langle X_1 \rangle _{Shu^t_n(G_1)})$ and $\phi(v_j)=v$ for any $v_j \in V(\langle X_2 \rangle _{Shu^t_n(G_2)})$.
    Let $x_i,y_i \in X_1$ and $u_j,v_j \in X_2$, such that $x_iy_i \in E(\langle X_1 \rangle _{Shu^t_n(G_1)})$ and $u_jv_j \in E(\langle X_2 \rangle _{Shu^t_n(G_2)})$. We have $x,y \in V(G_1)$ and $u,v \in V(G_2)$, and it also holds that $x_iy_i \in E(Shu^t_n(G_1)$ and $u_jv_j \in E(Shu^t_n(G_2)$. Since $i,j \in \{t+1,t+2,\dots,n\}$, it must be $\psi(x_i)\psi(y_i)=xy \in E(G_1)$ and $\phi(u_j)\phi(v_j)=uv \in E(G_2)$. Conversely, let $x_i,y_i \in X_1$ and $u_j,v_j \in X_2$, such that $xy \in E(G_1)$ and $uv \in E(G_2)$. Based on the definition of adjacency in the graphs $Shu^t_n(G_1)$ and $Shu^t_n(G_2)$, we obtain $x_iy_i \in E(Shu^t_n(G_1)$ and $u_jv_j \in E(Shu^t_n(G_2)$. In other word, $x_iy_i \in E(\langle X_1 \rangle _{Shu^t_n(G_1)})$ and $u_jv_j \in E(\langle X_2 \rangle _{Shu^t_n(G_2)})$
    Hence, 
    $$\langle X_1 \rangle _{Shu^t_n(G_1)} \cong G_1 \text{ and } \langle X_2 \rangle _{Shu^t_n(G_2)} \cong G_2.$$ 
    
    Since $Shu^t_n(G_1) \cong Shu^t_n(G_2)$, there exists a graph isomorphism 
    $$f: Shu^t_n(G_1) \to Shu^t_n(G_2).$$
    Observe that:
    \begin{align*}
        deg_{Shu^t_n(G_1)}(z_i)&=deg_{Shu^t_n(G_2)}(z_i)=k, \text{ for } 1 \leq i \leq t,\\
        deg_{Shu^t_n(G_1)}(z_i)&=deg_{Shu^t_n(G_2)}(z_i)=k+1, \text{ for } t+1 \leq i \leq n,\\
        deg_{Shu^t_n(G_1)}(u_i)&=n (deg_{G_1}(u)) + k, \text{ for } u \in V(G_1) \text{ and } 1 \leq i \leq t,\\
        deg_{Shu^t_n(G_1)}(u_i)&=n (deg_{G_1}(v)) + k+1, \text{ for } u \in V(G_1) \text{ and } t+1 \leq i \leq n,\\
        deg_{Shu^t_n(G_2)}(v_i)&=n (deg_{G_2}(v)) + k, \text{ for } v \in V(G_2) \text{ and } 1 \leq i \leq t,\\
        deg_{Shu^t_n(G_2)}(v_i)&=n (deg_{G_2}(v)) + k+1, \text{ for } v \in V(G_2) \text{ and } t+1 \leq i \leq n,
    \end{align*}
    therefore $$f(\{z_i: i=1,2,\dots,t\})=\{z_i: i=1,2,\dots,t\}$$ and $$f(\{z_i: i=t+1,t+2,\dots,n\})=\{z_i: i=t+1,t+2,\dots,n\}.$$
    Furthermore,
    \begin{align*}
        N_{Shu^t_n(G_1)}(z_i)&=\{u_i: u \in V(G_1)\}, \text{ for } 1 \leq i \leq t,\\
        N_{Shu^t_n(G_1)}(z_a)&=\{z_{n+t+1-a}, u_{n+t+1-a}: u \in V(G_1)\}, \text{ for } t+1 \leq a \leq n,\\
        N_{Shu^t_n(G_2)}(z_j)&=\{v_j: v \in V(G_2)\}, \text{ for } 1 \leq j \leq t,\\
        N_{Shu^t_n(G_2)}(z_b)&=\{z_{n+t+1-b}, v_{n+t+1-b}: v \in V(G_2)\}, \text{ for } t+1 \leq b \leq n.
    \end{align*}
    Consider $u_n \in V(Shu^t_n(G_1))$ for any $u \in V(G_1)$. Since $u_nz_{t+1} \in E(Shu^t_n(G_1))$, we have $f(u_n)f(z_{t+1}) \in E(Shu^t_n(G_2))$. Assume that $f(u_n)=v_m$ and $f(z_{t+1})=z_{n+t+1-m}$ with $t+1 \leq m \leq n$.
    Consequently, $$f(\{u_n: u \in V(G_1)\})=\{v_m: v \in V(G_2)\}.$$ Let $X_1=\{u_n: u \in V(G_1)\}$ and $X_2=\{v_m: v \in V(G_2)\}$. As $f$ is a graph isomorphism, it follows that $\langle X_1 \rangle_{Shu^t_n(G_1)} \cong \langle X_2 \rangle_{Shu^t_n(G_2)}$. 
    Hence,
    $$G_1 \cong \langle X_1 \rangle_{Shu^t_n(G_1)} \cong \langle X_2 \rangle_{Shu^t_n(G_2)} \cong G_2.$$
    For the converse, if $G_1 \cong G_2$, it follows trivially that $Shu^t_n(G_1) \cong Shu^t_n(G_2)$.
\end{proof}

\begin{theorem}\label{teo_isoZp1p2p3}
    For any distinct prime numbers $p_1,p_2,p_3$ and natural numbers $n_1,n_2,n_3$, the following holds
    $$Cl_2(\mathbb{Z}_{\tiny{p_1^{n_1}p_2^{n_2}p_3^{n_3}}}) \cong Shu^{|U'(\mathbb{Z}_{\tiny{p_1^{n_1}p_2^{n_2}p_3^{n_3}}})|}_{|U(\mathbb{Z}_{\tiny{p_1^{n_1}p_2^{n_2}p_3^{n_3}}})|}(G_1),$$
    where the structure of the graph $G_1$ is shown in the Figure \ref{grafg1}:
    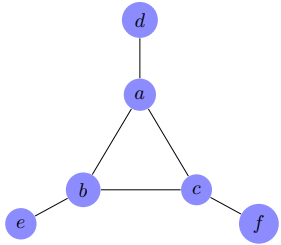
\begin{figure}[H]
		\begin{center} 
            \resizebox{0.3\textwidth}{!}{\begin{tikzpicture}
			[scale=1,auto=center,every node/.style={circle, fill=blue!45}] 
			\node (v1) at (0,0.68) {$a$};
			\node (v2) at (-1,-1) {$b$};
			\node (v3) at (1,-1) {$c$};
                \node (v4) at (0,2) {$d$};
                \node (v5) at (-2.1,-1.6) {$e$};
			\node (v6) at (2.1,-1.6) {$f$};

			\draw (v1) -- (v2);
                \draw (v2) -- (v3);
                \draw (v1) -- (v3);
                \draw (v1) -- (v4);
                \draw (v2) -- (v5);
                \draw (v3) -- (v6);
            \end{tikzpicture}}
            \end{center}\label{grafg1}\caption{Graph $G_1$}
	\end{figure}
\end{theorem}
\begin{proof}
    We know that
    \begin{align*}
        V(Cl_2(\mathbb{Z}_{\tiny{p_1^{n_1}p_2^{n_2}p_3^{n_3}}}))&=\{(0,0,1),(0,1,0),(1,0,0),(0,1,1),(1,0,1),(1,1,0),\\
        &\hspace{0.6cm} (1,1,1)\}\times \left(U'(\mathbb{Z}_{\tiny{p_1^{n_1}p_2^{n_2}p_3^{n_3}}}) \cup U''(\mathbb{Z}_{\tiny{p_1^{n_1}p_2^{n_2}p_3^{n_3}}})\right).
    \end{align*}
    Assume that $t=|U'(\mathbb{Z}_{\tiny{p_1^{n_1}p_2^{n_2}p_3^{n_3}}})|$ and $k=|U(\mathbb{Z}_{\tiny{p_1^{n_1}p_2^{n_2}p_3^{n_3}}})|$.
    Let $$U(\mathbb{Z}_{\tiny{p_1^{n_1}p_2^{n_2}p_3^{n_3}}})=\{u_1,u_2,\dots,u_k\},$$ with $u_1,u_2,\dots,u_t \in U'(\mathbb{Z}_{\tiny{p_1^{n_1}p_2^{n_2}p_3^{n_3}}}))$ and $u_{t+1},u_{t+2},\dots,u_{k} \in U''(\mathbb{Z}_{\tiny{p_1^{n_1}p_2^{n_2}p_3^{n_3}}})$, such that 
    $$u_{t+i}u_{k+1-i} \text{, for all } 1 \leq i \leq k-t.$$ Construct a bijective function $f:V(Cl_2(\mathbb{Z}_{\tiny{p_1^{n_1}p_2^{n_2}p_3^{n_3}}})) \to V(Shu^t_k(G_1))$, where for each $i \in \{1,2,\dots,k\}$\\
    \begin{tabular}{ccc}
    $f((0,0,1),u_i)=a_i$, & $f((0,1,0),u_i)=b_i$, & $f((1,0,0),u_i)=c_i$,\\
    $f((0,1,1),u_i)=f_i$, & $f((1,0,1),u_i)=e_i$, & $f((1,1,0),u_i)=d_i$,\\
    $f((1,1,1),u_i)=z_i$.
    \end{tabular}
    
    Let $w_1=(x,u_i)$, $w_2=(y,u_j)$ be arbitrary vertices of graph $Cl_2(\mathbb{Z}_{\tiny{p_1^{n_1}p_2^{n_2}p_3^{n_3}}})$. If $w_1w_2=(x,u_i)(y,u_j) \in E(Cl_2(\mathbb{Z}_{\tiny{p_1^{n_1}p_2^{n_2}p_3^{n_3}}}))$, then $xy=0$ or $u_iu_j=1$. Two cases are considered: \begin{enumerate}
        \item Case 1: $xy=0$. By examining its idempotent elements that form zero divisors with each other, we obtain 
        \begin{align*}
            \{x,y\}&\in\{\{(0,0,1),(0,1,0)\}, \{(0,0,1),(1,0,0)\},\{(0,1,0),(1,0,0)\}\\
            &\hspace{0.7cm}\{(0,0,1),(1,1,0)\},\{(0,1,0),(1,0,1)\},\{(1,0,0),(0,1,1)\}\}.
        \end{align*}
        Consequently
        \begin{align*}
            \{f(w_1),f(w_2)\}&\in\{\{a_i,b_j\}, \{a_i,c_j\},\{b_i,c_j\},\{a_i,d_j\},\{b_i,e_j\},\{c_i,f_j\}\}.
        \end{align*}
        From the adjacency properties of the graph $G_1$, it can be concluded that $f(w_1)f(w_2) \in E(Shu^t_k(G_1))$.
        \item Case 2: $u_iu_j=1$. Based on the assumption regarding the elements of $U(Cl)$ stated earlier, it must be $i=j$ for $1 \leq i \leq t$, or $j=k+t+1-i$ for $t+1 \leq i \leq k$. Hence, $f(w_1)f(w_2) \in E(Shu^t_k(G_1))$.
    \end{enumerate}
    In contrast, let $f(w_1) = p_i$, $f(w_2) = q_j$. If $f(w_1)f(w_2) = p_i q_j \in E(Shu^t_k(G_1))$, the following possibilities emerge:
    \begin{enumerate} 
    \item If $pq \in E(G_1)$, then using the same method as in Case 1 above, we obtain $w_1 w_2 \in E(Cl_2(\mathbb{Z}_{\tiny{p_1^{n_1}p_2^{n_2}p_3^{n_3}}}))$. 
    \item If $i = j$ with $1 < i < t$ or $j = k + t + 1 - i$ with $t + 1 < i < k$, then by following the reasoning in Case 2 above, we obtain $w_1 w_2 \in E(Cl_2(\mathbb{Z}_{\tiny{p_1^{n_1}p_2^{n_2}p_3^{n_3}}}))$. \end{enumerate}
\end{proof}

From Theorem \ref{teo_isoZp1p2p3}, the following proposition can be directly derived.

\begin{proposition}\label{propisoZp1p2p3}
    For any distinct prime numbers $p_1,p_2,p_3$ and natural numbers $n_1,n_2,n_3$, the following holds
    $$Cl_2(\mathbb{Z}_{p_1^{n_1}p_2^{n_2}p_3^{n_3}}) \cong \begin{cases}
        Shu^4_{\left(p_2^{n_2}-p_2^{n_2-1}\right)\left(p_3^{n_3}-p_3^{n_3-1}\right)}(G_1), & \text{if } p_1^{n_1}=2,\\
        Shu^8_{2\left(p_2^{n_2}-p_2^{n_2-1}\right)\left(p_3^{n_3}-p_3^{n_3-1}\right)}(G_1), & \text{if } {p_1}^{n_1}=4,\\
        Shu^{16}_{m}(G_1), & \text{if } p_1=2, n_1 \geq 3,\\
        Shu^8_{m}(G_1), & \text{if } p_1,p_2,p_3 \neq 2,
    \end{cases}$$
    where $m=\left(p_1^{n_1}-p_1^{n_1-1}\right)\left(p_2^{n_2}-p_2^{n_2-1}\right)\left(p_3^{n_3}-p_3^{n_3-1}\right)$.
\end{proposition}
\begin{proof}
    Based on Theorem \ref{teo_isoZp1p2p3}, we have
    $$Cl_2(\mathbb{Z}_{p_1^{n_1}p_2^{n_2}p_3^{n_3}}) \cong Shu^{|U'(\mathbb{Z}_{p_1^{n_1}p_2^{n_2}p_3^{n_3}})|}_{|U(\mathbb{Z}_{p_1^{n_1}p_2^{n_2}p_3^{n_3}})|}(G_1).$$
    From Lemma \ref{lemma_U'U''(R)}, we obtain four possibilities.
    \begin{enumerate}
        \item Case $p_1^{n_1}=2$. So, $p_2,p_3 \neq 2$. We get $$|U(\mathbb{Z}_{p_1^{n_1}p_2^{n_2}p_3^{n_3}})|=\left(p_2^{n_2}-p_2^{n_2-1}\right)\left(p_3^{n_3}-p_3^{n_3-1}\right)$$ and
        \begin{align*}
            \left(|U'(\mathbb{Z}_{p_1^{n_1}})|=1 \text{, }  |U'(\mathbb{Z}_{p_2^{n_2}})|=2 \text{, }  |U'(\mathbb{Z}_{p_3^{n_3}})|=2 \right) \Rightarrow  |U'(\mathbb{Z}_{p_1^{n_1}p_2^{n_2}p_3^{n_3}})|=4.
        \end{align*}
        
        \item Case $p_1^{n_1}=4$. Hence, $p_2,p_3 \neq 2$. We get $$|U(\mathbb{Z}_{p_1^{n_1}p_2^{n_2}p_3^{n_3}})|=2\left(p_2^{n_2}-p_2^{n_2-1}\right)\left(p_3^{n_3}-p_3^{n_3-1}\right)$$ and
        \begin{align*}
            \left(|U'(\mathbb{Z}_{p_1^{n_1}})|=2 \text{, }  |U'(\mathbb{Z}_{p_2^{n_2}})|=2 \text{, }  |U'(\mathbb{Z}_{p_3^{n_3}})|=2 \right) \Rightarrow  |U'(\mathbb{Z}_{p_1^{n_1}p_2^{n_2}p_3^{n_3}})|=8.
        \end{align*}
        
        \item Case $p_1=2$ dan $n_1 \geq 3$. It must be $p_2,p_3 \neq 2$. So,
        \begin{align*}
            |U(\mathbb{Z}_{p_1^{n_1}p_2^{n_2}p_3^{n_3}})|=\left(p_1^{n_1}-p_1^{n_1-1}\right)\left(p_2^{n_2}-p_2^{n_2-1}\right)\left(p_3^{n_3}-p_3^{n_3-1}\right) \text{ dan }\\
           \left(|U'(\mathbb{Z}_{p_1^{n_1}})|=4 \text{, }  |U'(\mathbb{Z}_{p_2^{n_2}})|=2 \text{, }  |U'(\mathbb{Z}_{p_3^{n_3}})|=2 \right) \Rightarrow  |U'(\mathbb{Z}_{p_1^{n_1}p_2^{n_2}p_3^{n_3}})|=16.
        \end{align*}
       
        \item Case $p_1,p_2,p_3 \neq 2$. We have 
        \begin{align*}
            |U(\mathbb{Z}_{p_1^{n_1}p_2^{n_2}p_3^{n_3}})|=\left(p_1^{n_1}-p_1^{n_1-1}\right)\left(p_2^{n_2}-p_2^{n_2-1}\right)\left(p_3^{n_3}-p_3^{n_3-1}\right) \text{ dan }\\
           \left(|U'(\mathbb{Z}_{p_1^{n_1}})|=2 \text{, }  |U'(\mathbb{Z}_{p_2^{n_2}})|=2 \text{, }  |U'(\mathbb{Z}_{p_3^{n_3}})|=2 \right) \Rightarrow  |U'(\mathbb{Z}_{p_1^{n_1}p_2^{n_2}p_3^{n_3}})|=8.
        \end{align*}
    \end{enumerate}
\end{proof} 

In the previous discussion on the clean graph $Cl_2(\mathbb{Z}_{\tiny{p_1^{n_1}p_2^{n_2}p_3^{n_3}}})$, we observe the proof of Case 1 in Theorem \ref{teo_isoZp1p2p3}, its structure can be interpreted through the interaction between two nonzero idempotent elements whose products result in zero. More generally, for any ring $R$, the idempotent element $1_R$ clearly does not form a zero product with any other nonzero idempotent and therefore does not contribute to such interactions. Consequently, it is natural to restrict out attention to the nontrivial idempotent elements. This observation is closely related to the concept of the idempotent graph $I(R)$, where edges reflect zero products between pairs of nontrivial idempotent. Motivated by this relation, we now state the following general theorem.

\begin{theorem}\label{teoremCl2Shutn}
    Let $R$ be an arbitrary ring with an identity element. It holds that
    $$Cl_2(R) \cong Shu^{|U'(R)|}_{|U(R)|}(I(R)).$$
\end{theorem}
\begin{proof}
    Let
    \begin{align*}
        Id(R) \setminus \{0\} &= \{e_1=1, e_2, e_3, \dots, e_k\} \text{ and }\\
        U(R)&=\{u_1,u_2,\dots,u_t,u_{t+1},u_{t+2},\dots,u_{n}\} \text{, where}\\
        U'(R) &=\{u_1,u_2,\dots,u_t\} \text{ and } u_{i}u_{n+t+1-i}=u_{n+t+1-i}u_{i}=1 \\
        &\text{ for each } i=t+1,t+2,\dots,\frac{n+t}{2}.
    \end{align*}
    We get 
    \begin{align*}
        V(Cl_2(R))=\{(e_i,u_j): i=1,2,\dots,k, j=1,2,\dots,n\}.
    \end{align*}
    On the other hand, 
    \begin{align*}
        V(I(R))&=\{e_2,e_3,\dots,e_k\}, \text{ so}\\
        V(Shu^{|U'(R)|}_{|U(R)|}(I(R)))&=V(Shu^{t}_{n}(I(R)))\\
        &=\bigcup_{i=1}^n \{e_{1i}, v_i: v \in V(I(R))\}\\
        &=\{e_{11},e_{12},\dots,e_{1n}, e_{21},e_{22},\dots,e_{2n},\\
        & \hspace{0.6cm} e_{31},e_{32},\dots,e_{3n}, \dots,e_{k1}, e_{k2}, \dots, e_{kn}\}
    \end{align*}
    and
    \begin{align*}
        E(Shu^t_n(I(R)))= &\{e_{ia}e_{jb} : e_ie_j \in E(I(R)), a,b \in \{1,2,\dots,n\}\} \\
        &\cup \{e_{li}e_{ji}: l,j \in \{1,2,\dots,k\}, l\neq j, i \in\{1,2,\dots,t\}\}\\ 
        &\cup \Bigg\{e_{li}e_{j(n+t+1-i)}: l,j \in \{1,2,\dots,k\}, \\
        & \hspace{3.4 cm}i \in\left\{t+1,t+2,\dots,\frac{n+t}{2}\right\}\Bigg\}.
    \end{align*}
    Consequently, we obtain
    $|V(Cl_2(R))|=nk=|V(Shu^{t}_{n}(I(R)))|$. Next, define the function $f: V(Cl_2(R)) \to V(Shu^{t}_{n}(I(R)))$, where
    \begin{align*}
        f(e_j,u_i)=e_{ji}
    \end{align*}
    for all $i=1,2,\dots,n$ and $j=1,2,\dots,k$.
    Let $(e_{j_1},u_{i_1}),(e_{j_2},u_{i_2}) \in V(Cl_2(R))$ such that $f((e_{j_1},u_{i_1}))=f((e_{j_2},u_{i_2}))$, we get
    \begin{align*}
        e_{j_1i_1}=e_{j_2i_2} \iff j_1=j_2 \text{ and } i_1=i_2 \iff (e_{j_1},u_{i_1})=(e_{j_2},u_{i_2}).
    \end{align*}
    Hence, $f$ is bijective function.

    Let $(e_{j_1},u_{i_1}),(e_{j_2},u_{i_2}) \in V(Cl_2(R))$ such that $$(e_{j_1},u_{i_1})(e_{j_2},u_{i_2}) \in E(Cl_2(R)).$$ We have two possibilites, $e_{j_1}e_{j_2}=e_{j_2}e_{j_1}=0$ or $u_{i_1}u_{i_2}=u_{i_2}u_{i_1}=1$.
    \begin{enumerate}
        \item If $e_{j_1}e_{j_2}=e_{j_2}e_{j_1}=0$, then $e_{j_1}e_{j_2} \in E(I(R))$. So, $$f((e_{j_1},u_{i_1}))f((e_{j_2},u_{i_2}))=e_{j_1i_1}e_{j_2i_2} \in E(Shu^t_n(I(R))).$$
        \item If $u_{i_1}u_{i_2}=u_{i_2}u_{i_1}=1$, then $i_2=n+t+1-i_1$ with $i_1 \in \{t+1,t+2,\dots,n\}$ or $i_1=i_2 \in \{1,2,\dots,t\}$. Consequently,
        $$f((e_{j_1},u_{i_1}))f((e_{j_2},u_{i_2}))=e_{j_1i_1}e_{j_2(n+t+1-i_1)} \in E(Shu^t_n(I(R)))$$
        or
        $$f((e_{j_1},u_{i_1}))f((e_{j_2},u_{i_2}))=e_{j_1i_1}e_{j_2i_1} \in E(Shu^t_n(I(R)))$$
    \end{enumerate}
    Let $(e_{j_1},u_{i_1}),(e_{j_2},u_{i_2}) \in V(Cl_2(R))$ such that $$f(e_{j_1},u_{i_1})f(e_{j_2},u_{i_2}) \in E(Shu^t_n(I(R))).$$ In other words, $e_{j_1i_1}e_{j_2i_2} \in E(Shu^t_n(I(R)))$. Based on the definition of the elements in $E(Shu^t_n(I(R)))$, the following three possibilities are obtained:
    \begin{enumerate}
        \item If $e_{j_1}e_{j_2} \in I(R)$, then $e_{j_1}e_{j_2}=e_{j_2}e_{j_1}=0$. Hence, $(e_{j_1},u_{i_1})(e_{j_2},u_{i_2}) \in E(Cl_2(R))$.
        \item If $i_1=i_2 \in \{1,2,\dots,t\}$ and $e_{j_1} \neq e_{j_2}$, we have $u_{i_1}u_{i_2}=u_{i_2}u_{i_1}=1$. So, $(e_{j_1},u_{i_1})(e_{j_2},u_{i_2}) \in E(Cl_2(R))$.
        \item If $i_2=n+t+1-i_1 \in \{t+1,t+2,\dots,n\}$, then 
        $$u_{i_1}u_{i_2}=u_{i_1}u_{n+t+1-i_1}=1=u_{n+t+1-i_1}u_{i_1}=u_{i_2}u_{i_1}.$$
        Hence, $(e_{j_1},u_{i_1})(e_{j_2},u_{i_2}) \in E(Cl_2(R))$.
    \end{enumerate}
\end{proof}

After establishing the general structure theorem for clean graphs by relating them to the idempotent graph of a ring, we can now apply this result to determine the clean graph structures over rings such as $\mathbb{Z}_{\tiny{p_1^{n_1}p_2^{n_2}p_3^{n_3}p_4^{n_4}}}$ and more broadly over $\mathbb{Z}_n$ for any positive integer $n$.

\begin{proposition}\label{propisoZp1p2p3p4}
    For any distinct prime numbers $p_1,p_2,p_3,p_4$ and natural numbers $n_1,n_2,n_3,n_4$, the following holds
    $$Cl_2(\mathbb{Z}_{p_1^{n_1}p_2^{n_2}p_3^{n_3}p_4^{n_4}}) \cong \begin{cases}
        Shu^8_{m}(G_2), & \text{ if } p_1^{n_1}=2,\\
        Shu^{16}_{2m}(G_2), & \text{ if } {p_1}^{n_1}=4,\\
        Shu^{32}_{\left(p_1^{n_1}-p_1^{n_1-1}\right) m}(G_2), & \text{ if } p_1=2, n_1 \geq 3,\\
        Shu^{16}_{\left(p_1^{n_1}-p_1^{n_1-1}\right) m}(G_2), & \text{ if } p_1,p_2,p_3,p_4 \neq 2,
    \end{cases}$$
    where $m=\left(p_2^{n_2}-p_2^{n_2-1}\right)\left(p_3^{n_3}-p_3^{n_3-1}\right)\left(p_4^{n_4}-p_4^{n_4-1}\right)$ and the structure of the graph $G_2$ is shown in the following figure:
    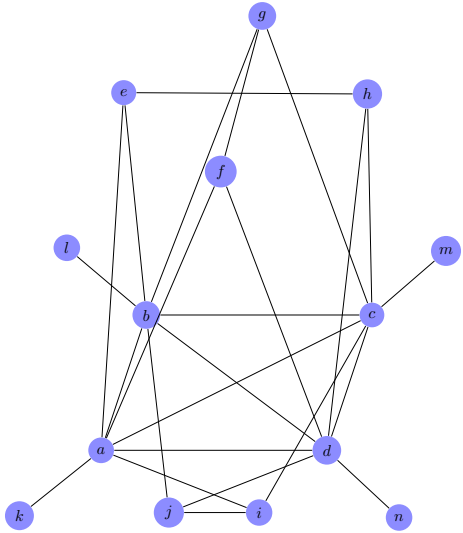
\begin{figure}[H]
    \begin{center} 
            \resizebox{0.5\textwidth}{!}{\begin{tikzpicture}
			[scale=1,auto=center,every node/.style={circle, fill=blue!45}] 
			\node (v1) at (-4,-1) {$a$};
			\node (v2) at (-3,2) {$b$};
			\node (v3) at (2,2) {$c$};
                \node (v4) at (1,-1) {$d$};
                \node (v5) at (-3.5,6.93) {$e$};
			\node (v6) at (-1.35,5.18) {$f$};
                \node (v7) at (-0.43,8.63) {$g$};
                \node (v8) at (1.9,6.9) {$h$};
			\node (v9) at (-4.76,3.49) {$l$};
                \node (v10) at (-5.81,-2.45) {$k$};
			\node (v11) at (3.64,3.42) {$m$};
                \node (v12) at (2.6, -2.49) {$n$};
                \node (v13) at (-0.5, -2.38) {$i$};
			\node (v14) at (-2.5, -2.38) {$j$};

			\draw (v1) -- (v2);
                \draw (v2) -- (v3);
                \draw (v3) -- (v4);
                \draw (v1) -- (v4);
                \draw (v1) -- (v3);
                \draw (v2) -- (v4);
                \draw (v1) -- (v5);
                \draw (v2) -- (v5);
                \draw (v1) -- (v6);
                \draw (v4) -- (v6);
                \draw (v2) -- (v7);
                \draw (v3) -- (v7);
                \draw (v3) -- (v8);
                \draw (v4) -- (v8);
                \draw (v1) -- (v10);
                \draw (v2) -- (v9);
                \draw (v3) -- (v11);
                \draw (v4) -- (v12);
                \draw (v1) -- (v13);
                \draw (v3) -- (v13);
                \draw (v2) -- (v14);
                \draw (v4) -- (v14);
                \draw (v5) -- (v8);
                \draw (v6) -- (v7);
                \draw (v13) -- (v14);
            \end{tikzpicture}}
            \end{center}\caption{Graph $G_2$}
\end{figure}
\end{proposition}
\begin{proof}
    Based on Theorem \ref{teoremCl2Shutn}, we have
    $$Cl_2(\mathbb{Z}_{p_1^{n_1}p_2^{n_2}p_3^{n_3}p_4^{n_4}}) \cong Shu^{|U'(\mathbb{Z}_{p_1^{n_1}p_2^{n_2}p_3^{n_3}p_4^{n_4}})|}_{|U(\mathbb{Z}_{p_1^{n_1}p_2^{n_2}p_3^{n_3}p_4^{n_4}})|}(I(\mathbb{Z}_{p_1^{n_1}p_2^{n_2}p_3^{n_3}p_4^{n_4}})).$$
    In this case, $G_2=I(\mathbb{Z}_{p_1^{n_1}p_2^{n_2}p_3^{n_3}p_4^{n_4}})$, with $a=(1,0,0,0)$, $b=(0,1,0,0)$, $c=(0,0,1,0)$, $d=(0,0,0,1)$, $e=(0,0,1,1)$, $f=(0,1,1,0)$, $g=(1,0,0,1)$, $h=(1,1,0,0)$, $i=(0,1,0,1)$, $j=(1,0,1,0)$, $k=(0,1,1,1)$, $l=(1,0,1,1)$, $m=(1,1,0,1)$, and $n=(1,1,1,0)$.
    From Lemma \ref{lemma_U'U''(R)}, we have four possibilities below.
    \begin{enumerate}
        \item Case $p_1^{n_1}=2$. So, $p_2,p_3,p_4 \neq 2$. We get $$|U(\mathbb{Z}_{p_1^{n_1}p_2^{n_2}p_3^{n_3}p_4^{n_4}})|=\left(p_2^{n_2}-p_2^{n_2-1}\right)\left(p_3^{n_3}-p_3^{n_3-1}\right)\left(p_4^{n_4}-p_4^{n_4-1}\right)$$ and
        \begin{align*}
            \left(|U'(\mathbb{Z}_{p_1^{n_1}})|=1 \text{, }  |U'(\mathbb{Z}_{p_2^{n_2}})|=2 \text{, }  |U'(\mathbb{Z}_{p_3^{n_3}})|=2 \text{, }  |U'(\mathbb{Z}_{p_4^{n_4}})|=2 \right) \\
            \Rightarrow  |U'(\mathbb{Z}_{p_1^{n_1}p_2^{n_2}p_3^{n_3}p_4^{n_4}})|=8.
        \end{align*}
        
        \item Case $p_1^{n_1}=4$. Hence, $p_2,p_3,p_4 \neq 2$. So, $$|U(\mathbb{Z}_{p_1^{n_1}p_2^{n_2}p_3^{n_3}p_4^{n_4}})|=2\left(p_2^{n_2}-p_2^{n_2-1}\right)\left(p_3^{n_3}-p_3^{n_3-1}\right)\left(p_4^{n_4}-p_4^{n_4-1}\right)$$ and
        \begin{align*}
             \left(|U'(\mathbb{Z}_{p_1^{n_1}})|=2 \text{, }  |U'(\mathbb{Z}_{p_2^{n_2}})|=2 \text{, }  |U'(\mathbb{Z}_{p_3^{n_3}})|=2 \text{, }  |U'(\mathbb{Z}_{p_4^{n_4}})|=2 \right) \\
            \Rightarrow  |U'(\mathbb{Z}_{p_1^{n_1}p_2^{n_2}p_3^{n_3}p_4^{n_4}})|=16.
        \end{align*}
        
        \item Case $p_1=2$ and $n_1 \geq 3$. It must be $p_2,p_3,p_4 \neq 2$. So,
        \begin{align*}|U(\mathbb{Z}_{p_1^{n_1}p_2^{n_2}p_3^{n_3}p_4^{n_4}})|=\left(p_1^{n_1}-p_1^{n_1-1}\right)&\left(p_2^{n_2}-p_2^{n_2-1}\right)\\
        &\left(p_3^{n_3}-p_3^{n_3-1}\right)\left(p_4^{n_4}-p_4^{n_4-1}\right)\end{align*} and
        \begin{align*}
           \left(|U'(\mathbb{Z}_{p_1^{n_1}})|=4 \text{, }  |U'(\mathbb{Z}_{p_2^{n_2}})|=2 \text{, }  |U'(\mathbb{Z}_{p_3^{n_3}})|=2 \text{, }  |U'(\mathbb{Z}_{p_4^{n_4}})|=2 \right) \\
            \Rightarrow  |U'(\mathbb{Z}_{p_1^{n_1}p_2^{n_2}p_3^{n_3}p_4^{n_4}})|=32.
        \end{align*}
       
        \item Case $p_1,p_2,p_3,p_4 \neq 2$. We have 
        $$|U(\mathbb{Z}_{p_1^{n_1}p_2^{n_2}p_3^{n_3}p_4^{n_4}})|=\left(p_1^{n_1}-p_1^{n_1-1}\right)\left(p_2^{n_2}-p_2^{n_2-1}\right)\left(p_3^{n_3}-p_3^{n_3-1}\right)\left(p_4^{n_4}-p_4^{n_4-1}\right)$$ and
        \begin{align*}
            \left(|U'(\mathbb{Z}_{p_1^{n_1}})|=2 \text{, }  |U'(\mathbb{Z}_{p_2^{n_2}})|=2 \text{, }  |U'(\mathbb{Z}_{p_3^{n_3}})|=2 \text{, }  |U'(\mathbb{Z}_{p_4^{n_4}})|=2 \right) \\
            \Rightarrow  |U'(\mathbb{Z}_{p_1^{n_1}p_2^{n_2}p_3^{n_3}p_4^{n_4}})|=16.
        \end{align*}
    \end{enumerate}
\end{proof}

\begin{corollary}
    For any natural number $n$, there exist distinct prime numbers $p_1, p_2, \dots, p_k$ and positive integers $n_1, n_2, \dots, n_k$ such that $n = p_1^{n_1}p_2^{n_2}\dots p_k^{n_k}$. The following holds
    \begin{align*}
        Cl_2(\mathbb{Z}_n)\cong \begin{cases}
            Shu^{2^{k-1}}_{m}(I(\mathbb{Z}_n)), &\text{if } 2 \mid n \text{ and } 4 \nmid n\\
            Shu^{2^{k+1}}_m(I(\mathbb{Z}_n)), &\text{if } 8 \mid n\\
            Shu^{2^k}_m(I(\mathbb{Z}_n)), &\text{otherwise}
        \end{cases}
    \end{align*}
    where $m=(p_1^{n_1}-p_1^{n_1-1})(p_2^{n_2}-p_2^{n_2-1})\dots(p_k^{n_k}-p_k^{n_k-1})$.
\end{corollary}


\section*{Conclusion}
This study explored the structure of clean graphs over the ring $\mathbb{Z}_n$, focusing on the subgraph $Cl_2(\mathbb{Z}_n)$ formed by pairs of nonzero idempotents and units. We examined its relationship with the idempotent graph $I(\mathbb{Z}_n)$ and found a strong structural correspondence between the clean graph over ring and their idempotent graph. The results highlight how the algebraic properties of idempotents and units in a ring shape the topology of their associated graphs. These insights enhance the theoretical framework for algebraic graphs and support further investigations into their structural and combinatorial properties.

\section*{Acknowledgement}
Sincere appreciation is extended to Professor Cihat Abdioğlu for introducing the research topic and providing foundational insights that guided the development of this work. Gratitude is also expressed to Professor Ahmad Erfanian for offering valuable comments and suggestions that further refined the study. This research was carried out during the graduate study of the first author at Universitas Gadjah Mada (UGM), Yogyakarta, Indonesia, with support from the Department of Mathematics, UGM, and funding provided by the Ministry of Research, Technology, and Higher Education of the Republic of Indonesia (Kementerian Riset, Teknologi, dan Pendidikan Tinggi) through the PMDSU (Program Magister Menuju Doktor untuk Sarjana Unggul) Scholarship, 2024–2028.

\end{document}